\documentclass[11pt, a4paper, oneside, reqno]{amsart}
\usepackage{amsmath}

\usepackage[numbered]{bookmark}
\usepackage{changepage} 
\usepackage[utf8]{inputenc}
\usepackage{graphics, graphicx}
\usepackage{amsmath, amsfonts, amsthm, amssymb}
\usepackage{mathrsfs, mathtools, mathabx}
\usepackage{verbatim} 
\usepackage{enumerate}
\usepackage{xcolor}
\usepackage{cite}
\usepackage{array}
\usepackage[T1]{fontenc}
\usepackage{tikz}

\newtheorem{theorem}{Theorem}[section]

\newtheorem{lemma}[theorem]{Lemma}
\newtheorem{proposition}[theorem]{Proposition}
\newtheorem{corollary}[theorem]{Corollary}

\newtheorem{remark}[theorem]{Remark}
\newtheorem{claim}{Claim}
\newtheorem{example}[theorem]{Example}

\newcommand{\A}{\mathcal{A}} 
\newcommand{\B}{\mathcal{B}}

\newcommand{\Z}{\mathbb{Z}} 
\newcommand{\N}{\mathbb{N}}

\newcommand{\R}{\mathbb{R}}

\newcommand{\I}{\mathcal{I}}

\date{}
\title	[Linear cocycles]{Ergodic properties of infinite extension of symmetric interval exchange transformations}
\author[P.~Berk]{Przemys\l aw Berk}
\address{Faculty of Mathematics and Computer Science, Nicolaus
	Copernicus University, ul. Chopina 12/18, 87-100 Toru\'n, Poland}
\email{zimowy@mat.umk.pl}
\author[F.~Trujillo]{Frank Trujillo}
\address{Centre de Recerca Matemàtica, 08193 Bellaterra,
Barcelona, Spain}
\email{ftrujillo@crm.cat}
\author[H.~Wu]{Hao Wu}
\address{Institut für Mathematik, Universität Zürich, Winterthurerstrasse 190, CH-8057 Zürich, Switzerland}
\email{hao.wu@math.uzh.ch}
\begin{document}
	
	\maketitle

 \begin{abstract}
We prove that skew products with the cocycle given by the function $f(x)=a(x-1/2)$ with $a\neq 0$ are ergodic for every ergodic symmetric IET in the base, thus giving the full characterization of ergodic extensions in this family. Moreover, we prove that under an additional natural assumption of unique ergodicity on the IET, we can replace $f$ with any differentiable function with a non-zero sum of jumps. Finally, by considering weakly mixing IETs instead of just ergodic, we show that the skew products with cocycle given by $f$ have infinite ergodic index.
\end{abstract}

\section{Main results}
Let $I$ be a bounded interval, equipped with the Borel $\sigma$-algebra and Lebesgue measure $\lambda_{I}$. Let $T=(\pi,\lambda)$ be an interval exchange transformation given by a permutation $\pi\in S_0^\mathcal A$
and a length vector $\lambda\in\Lambda^{\mathcal A,I}$
{(see Section \ref{sc:IETs} for precise definitions of these objects)}. It is not difficult to see that $T$ preserves $\lambda_{I}$. We say that a permutation is \emph{symmetric} if and only if for any $i = 1, \dots, d$ $\pi_1\circ\pi_0^{-1}(i)=d+1-i$. 

The main objects of study in this article are (real-valued) skew products over interval exchange transformations (IETs). More precisely, if $(X,\mathcal B,\mu)$ is a probability Borel space and $f:X\to \R$ is such that $\int_X f(x)\,d\mu(x)=0$, then a \emph{skew product} $T_f:X\times \R\to X\times \R$ over a measure-preserving map $(X,\mathcal B,\mu,T)$, is a transformation given by
\[
T_f(x,r):=(T(x),x+f(x)).
\]
We will refer to $f$ as a \emph{cocycle}. It is not difficult to see, that $T_f$ preserves the product measure of $\mu$ on $I$ and the Lebesgue measure $\lambda_{\R}$ on $\R$. We will investigate the ergodic properties of $T_f$ with respect to the measure $\lambda_I\otimes\lambda_{\R}$ when $T$ is either an IET or, more generally, a product of $n \ge 2$ copies of an IET. 

\begin{theorem}\label{thm: ergo}
	Let $T$ be an ergodic symmetric IET on $I = [0, 1)$ and let $f(x)=a(x-\tfrac{1}{2})$ for some $a\in\R\setminus\{0\}$. Then, the skew product $T_f:I\times\R\to I\times \R$ is ergodic w.r.t. $\lambda_I\otimes\lambda_{\R}$.
\end{theorem}
The exceptionality of the above theorem comes from the fact {that we only assume the ergodicity of the IET with respect to $\lambda_I$ (in contrast to many results in the theory of IETs and of skew-products over IETs where generic conditions related to the Rauzy-Veech renormalization procedure are often imposed). Moreover, this assumption is necessary since we can associate to any non-trivial $T$-invariant set $A \subseteq I$ the non-trivial $T_f$-invariant set $A \times \R \subseteq I \times \R$.} Theorem \ref{thm: ergo} gives thus a full characterization of the ergodic skew products over symmetric IETs with cocycle of the form $a(x-\tfrac{1}{2})$, for some $a \in \R\,\setminus\, \{0\}$. 

{The skew products over IETs were already researched with various types of cocycles, although not under such weak assumptions. Recently, the first and second author in \cite{berk_ergodicity_2024} proved that for almost every symmetric IET on $[0,1)$ and a cocycle $f(x)=\chi_{[0,1/2)}-\chi_{[1/2,1)}$ the skew product is ergodic. The final arguments in the proof of Theorem \ref{thm: ergo} are partially inspired by this paper. For linear cocycles, the most relevant is the work of Conze and Fr\k{a}czek in \cite{conze_cocycles_2011}, where the authors studied piece-wise linear cocycles over IETs of periodic type. However, there are only countably many such IETs. 

There are very few results in the literature concerning any ergodic IET. It is worth mentioning here the article \cite{katok_interval_1980} by Katok, where he proved that every IET is partially rigid. A variation of his construction of Rokhlin towers serves later in the proof of Theorem \ref{thm: ergo} to construct partially rigid towers needed to establish the ergodicity of skew products.}

If we only assume the ergodicity of the underlying IET, one of the main obstacles we face is the impossibility of excluding IETs with \emph{connections} {(see Section \ref{sc:IETs} for a precise definition). Let us point out that the existence of connections does not exclude ergodicity. {Indeed, perhaps the most relevant to this article is the example given in \cite{berk_ergodicity_2024} which was a symmetric IET of the interval $[0,1)$ with $\tfrac{1}{2}$ as a discontinuity. There it served as an example of IET taken as a base of a non-ergodic skew product. Here such examples are also covered in Theorem \ref{thm: ergo}. }By dealing with symmetric IETs with connections, we obtain interesting side results on their ergodic properties (see Corollary \ref{cor: emptyconnection} and Corollary \ref{cor: evalue-1}). }

{
Since we cannot use the ergodic properties of Rauzy-Veech induction due to the 
presence of connections, we have to tackle some issues that are usually not a 
problem if one wants to obtain a result only on a full-measure set of IETs. 
Notice that, given an IET, we can always induce on a subinterval and obtain 
another IET but, in general, we do not have control over the combinatorial 
properties of the induced map. Hence, a major step towards proving our main 
result is showing that if we choose the induction interval properly, then the 
induced transformation is a symmetric IET as long as the initial IET $T$ is 
symmetric. This is the content of Proposition \ref{prop:induced_symmetric}. The 
proof of this key property relies on another result that we would like to 
highlight and which generalizes a well-known property of the Rauzy-Veech 
induction, namely, the existence of neighborhoods (simplices) around almost 
every IET so that, for any IET in this neighborhood, the induced map on certain 
\emph{dynamically defined} induction intervals (given by a fixed number of 
iterations of the Rauzy-Veech procedure) leads to the same combinatorics and 
the same \emph{Rokhlin tower decomposition} of the initial intervals (see 
Section \ref{sc:induced_IETs}). We formally state this in Proposition 
\ref{prop: unwinding} (see also Remark \ref{rmk:simplex}). 


}

By imposing an additional generic condition on the IET, we can largely increase the family of cocycles for which we can deduce the ergodicity of the skew product. 

\begin{theorem}\label{thm: uniqergo}
Let $T$ be a uniquely ergodic symmetric IET on $I = [0, 1)$ and let $f(x)=a(x-\tfrac{1}{2})+f_0(x)$, for some $a\in\R\setminus\{0\}$ and some differentiable function $f_0$ satisfying $\int_I Df_0(x)\, dx=0$. Then, the skew product $T_f:I\times\R\to I\times \R$ is ergodic w.r.t. $\lambda_I\otimes\lambda_{\R}$.
\end{theorem}

Finally, one may ask about the ergodic index of the skew product under consideration. Recall that a measure-preserving transformation $(X,\mathcal B,\mu,T)$ has \emph{infinite ergodic index} if and only if for every $n\in\N$ the transformation $(X^{\times n},\mathcal B^{\otimes n}, \mu^{\otimes n},T^{\times n})$ is ergodic, {where the superscripts $\times n$ and $\otimes n$ denote $n$-fold products of the objects. 
}

If we consider this property for $T_f$ on $X\times \R$, it is easy to see that $T_f^{\times n}$ is a skew product over $T^{\times n}$ with the cocycle given by the function $f^{\times n}:X^{\times n}\to \R^n$, where $f^{\times n}(x_1,\ldots,x_n):=(f(x_1),\ldots,f(x_n))$. It is easy to see that a natural obstruction for having an infinite ergodic index is when $T^{\times k}$ is not ergodic, for some $k\in\N$. It turns out that in our case this is the only obstacle.

\begin{theorem}\label{thm: index}
	Let $T$ be a weakly mixing 
 symmetric IET on $I = [0, 1)$ and let $f(x)=a(x-\tfrac{1}{2})$, for some $a\in\R\setminus\{0\}$. Then, the skew product $T_f:I\times\R\to I\times \R$ has infinite ergodic index.
	\end{theorem}

\section{Interval exchange transformations}
\label{sc:IETs}


\subsection{Notations and basic properties} An \emph{interval exchange transformation (IET)} $T$ on a bounded interval $I$ is a piecewise linear bijection of $I$, with finite number of intervals of continuity, on which $T$ acts via translation. {For convenience and without loss of generality, we assume the interval $I$ to be of the form $[a, b)$, for some $a, b \in \R$, and the IETs to be right-continuous.}

 More precisely, there exists $\A$ is an alphabet of $d\in\N$ elements, a permutation $\pi=\binom{\pi_0}{\pi_1}$ with $\pi_0,\pi_1:\mathcal A\to\{1,\ldots,d\}$ and a collection of left closed and right open subintervals $\{I_{\alpha}\}_{\alpha\in\mathcal A}$ such that and $\bigsqcup_{\alpha\in\A} I_{\alpha}=I$, $T|_{I_{\alpha}}$ acts via translation, and $\pi_0$ and $\pi_1$ describe the order of intervals respectively before and after 
action of $T$. It is easy to check that $T$ preserves Lebesgue measure on $I$ and that the parameters $(\pi,\lambda)$ fully describe the dynamics of $T$, where $\lambda=[\lambda_\alpha]_{\alpha\in\A}:=[|I_{\alpha}|]_{\alpha\in\A}\in \Lambda^{\mathcal A,I}$ is the vector of lengths of intervals $I_{\alpha}$, with $\Lambda^{\mathcal A,I}:=\{\lambda\in\R^{\A}_{+}\mid \sum_{\alpha\in\A}\lambda_{\alpha}=|I| \}$. {Moreover, we always assume that $\pi$ is \emph{non-reducible}, that is
\[
\pi_1\circ\pi^{-1}_0\{1,\ldots,k\}=\{1,\ldots,k\}\ \Rightarrow\ k=d.
\]
Otherwise, we can decompose $T$ into two non-trivial IETs and consider their properties separately. We denote by $S_0^{\A}$ the set of all non-reducible permutations of alphabet $\A$.
}

For every $\alpha\in\A$, we denote by $\partial I_{\alpha}$ the left endpoint of $I_{\alpha}$ and by $c_{\alpha}$ its center point. We say that an IET $T$ has a \emph{connection} if there exist $\alpha,\beta\in\A$ with $\pi_0(\beta)\neq 1$, $\pi_1(\alpha)\neq 1$ and $n\in\N_+$ such that 
\[
T^{-n}(\partial I_\beta)=\partial I_{\alpha}.
\] 
By connection we often mean the orbit segment $\{T^{-k}(\partial I_\beta)\}_{k=-n,\ldots,0}$. 
If such connection exists, we denote $M(\beta)=M(T,\beta):=\min_{\alpha\in\A}\min\{n\in\N_+\mid T^{-n}(\partial I_\beta)=\partial I_{\alpha} \}$. Otherwise, we write $M(\beta)=\infty$. Similarly, we denote $N(\alpha)=N(T,\alpha):=\min_{\beta\in\A}\min\{n\in\N_+\mid T^{n}(\partial I_\alpha)=\partial I_{\beta} \}$ and write $N(\alpha)=\infty$ if such connection does not exist. {Note that we always have $T(\partial I_{\pi_1^{-1}(1)})=\partial I_{\pi_0^{-1}(1)}$, a trivial connection. Hence, we define $M(\pi_0^{-1}(1)):=1$ and $N(\pi_1^{-1}(1)):=1$.}

Note that the existence of a non-trivial connection implies that some 
non-trivial integer combination of lengths of exchanged intervals is equal to 
$0$. Thus, if the length vector is \emph{rationally independent}, i.e., if 
$\sum_{\alpha\in\A}r_{\alpha}\lambda_\alpha=0$ for some $(r_{\alpha})_{\alpha 
\in \A} \in\mathbb Q^{\A}$ implies that $r_{\alpha}=0$, for all $\alpha\in\A$, 
then there cannot be any connection. Hence, almost every IET has no 
connections. 

However, in this article we consider the class of \emph{all} ergodic IETs, and, let us recall, the existence of connections does not exclude ergodicity. Nevertheless, it is well-known that if all $\partial I_{\beta}$ are endpoints of connections, then such an IET cannot be ergodic. For the sake of completeness, let us provide a proof of this fact.

\begin{lemma}\label{lem: notallconnections}
	Let $T:I\to I$ be an interval exchange transformation given by a permutation $\pi$ and length vector $\lambda$. If for every $\beta\in \A\,\setminus\,\{\pi_0^{-1}(1)\}$ we have $M(\beta)<\infty$, then $T$ has only periodic orbits. {More precisely, the base interval $I$ can be decomposed in a finite number of periodic components, given by semi-closed intervals, such that the period is uniform on each of these components.}
\end{lemma}
\begin{proof}
	Note that by assumption the set of points 
	\[
	\{T^{n}(\partial I_{\alpha})\mid n\in\Z,\quad\alpha\in \A \}=\{T^{-n}(\partial I_{\alpha})\mid \alpha\in\A\quad\text{and}\quad 0\le n\le M(\alpha) \}
	\]
	 is finite. Consider the partition given by those points and let $[a,b)$ be an element of this partition. 
	
	Note that $T^n$ acts continuously on $[a,b)$ for all $n\in\N$. Indeed, the only possible points of discontinuity of $T$ are $\{\partial I_{\alpha}\}_{\alpha\in\A}$, hence, if for some $n\in\N$ the map $T^n$ did not act continuously on $[a,b)$, there would exist $\beta\in\mathcal A$ such that $\partial I_{\beta}\in T^{n-1}\big([a,b)\big)$, which contradicts the choice of $[a,b)$. In particular, it follows that $T^n\big([a,b)\big)$ is an interval, for any $n \in \N$.
	
 By Poincaré's recurrence theorem, there exists $N\in\N$ such that 
	\[
	T^{N}\big([a,b)\big)\cap [a,b)\neq \emptyset.
	\]
	This implies that $T^{N}\big([a,b)\big)=[a,b)$. Indeed, otherwise either $T^{-N}(a)$ or $T^{N}(a)$ belongs to $[a,b)$. Since $a$ is in the orbits of one of the points $\{\partial I_{\alpha}\}_{\alpha\in\A}$, this yields a contradiction.
	
	To sum up, $T^{N}\big([a,b)\big)=[a,b)$ and $T^N$ acts continuously on $[a,b)$. Since $T$ is a piecewise translation, then so is $T^N$. Thus $T^N|_{[a,b)}$ is the identity map on $[a,b)$, which finishes the proof. 
\end{proof}

\begin{remark}
	By proceeding symmetrically, one can replace in Lemma \ref{lem: notallconnections} the endpoints of connections with their initial points.
\end{remark}

One of the main consequences of the above lemma is the following.

\begin{corollary}\label{cor: existenceofinfiniteorbit}
	Assume that $T$ is an ergodic IET. Then there exists {$\beta \in\mathcal A \,\setminus\, \{\pi_0^{-1}(1)\}$ such that $M(\beta) = +\infty$. }
\end{corollary}
\begin{proof}
{Assume, for the sake of contradiction, that $T$ is ergodic but that the conclusion does not hold. Then, by Lemma \ref{lem: notallconnections}, there exists a non-empty semi-closed interval $[a, b) \subseteq I$ and $N \geq 1$ such that $T^N\mid_{[a, b)}$ is the identity map in $[a, b)$. Therefore the set $\bigcup_{i = 0}^{N - 1} T^i\big( \big[a, \tfrac{a + b}{2}\big]\big)$ is a non-trivial $T$-invariant set, which contradicts the ergodicity of $T$.}
\end{proof}

\subsection{Induced IETs}\label{sc:induced_IETs}
Throughout the proofs of the main results of this paper, we will often use the first return map of $T$ to a subinterval $J \subseteq I$, {which we denote by $T_J: J \to J$. More precisely, we define $T_J$ as $x \mapsto T^{r_J(x)}(x)$, where $r_J: J \to \N$ is given by
\[ r_J(x):= \min\{n \geq 1 \mid T^n(x) \in J\}.\]
We sometimes refer to $T_J$ as the \emph{induced map of $T$ to $J$}.


A priori the map $T_J$ is not necessarily well-defined for \emph{all} points in $J$, although Poincaré's recurrence theorem guarantees that $T_J$ is well-defined in a full Lebesgue measure subset of $J$. However, it is well known (see, e.g., \cite[\S 3]{veech_interval_1978}) that for any subinterval $J=[a_J,b_J)\subseteq I$ the induced map $T_J$ is an IET of at most $d+2$ intervals, where the possible discontinuities are given by preimages of the discontinuities of $T$ (at most $d - 1$ points) and of the endpoints of $J$ (at most $2$ points, not necessarily disjoint with the previous set). 

More precisely, the possible discontinuity points of $T_J$ are given by
\[\{T^{-m_{J,\alpha}}(\partial I_{\alpha})\}_{\alpha\in \A \,\setminus\, \{\pi_0^{-1}(1)\}}, \qquad\ m_{J,\alpha}:=\inf\{n\geq 0 \mid T^{-n}(\partial I_{\alpha})\in \mathring{J} \}, \quad \text{ for } \alpha \in \A,\]
together with 
\[T^{-m_l}(a_J), \qquad m_l :=\inf\{n\ge 0 \mid T^{-n}(a_J)\in \mathring{J}\},\]
if $a_J$ is different from the left endpoint of $I$, and
\[T^{-m_r}(b_J), \qquad m_r :=\inf\{n\ge 0 \mid T^{-n}(b_J)\in \mathring{J}\},\]
if $b_J$ is different from the right endpoint of $I$,} where the preimages for 
which $m_{J, \alpha}$ (resp. $m_l$ or $m_r$) is $+\infty$ are disregarded. 
However, note that if $T$ is minimal, which is the case if $T$ is ergodic (see 
Lemma \ref{lem: minimality}), all the above notions are finite.

Moreover, if $T$ is ergodic and $J=[a_J,b_J)\subseteq I$ is chosen so that 
\begin{equation}\label{eq: paramofJ}
\begin{gathered}
\text{$a_J=T^{m_0}(\partial I_{\alpha_J})$ and $b_J=T^{n_0}(\partial I_{\beta_J})$, \quad for some $\alpha_J,\beta_J\in\mathcal A$ and $m_0,n_0\in\Z$,} \\
T^{m}(\partial I_{\alpha_J})\notin J, \qquad \text{ for any } m \in \{0, \ldots, m_0\} \text{ with } m \neq m_0 \\
T^n(\partial I_{\beta_J})\notin J \qquad \text{ for any } n \in \{0, \ldots, n_0\} \text{ with } n \neq n_0,
\end{gathered}
\end{equation}
then the induced map $T_J$ can be seen as an IET of at most $d$ intervals. 
Indeed, in this case, the discontinuities of $T_J$ belong to the set $
\{T^{-m_{J,\alpha}}(\partial I_{\alpha})\}_{\alpha\in \A\,\setminus\,\{\pi_0^{-1}(1)\}}.$ Analogously, if $J$ is of the form \eqref{eq: paramofJ} the discontinuities of $T_J^{-1}$ are contained in 
\[
\{T^{n_{J,\alpha}}(\partial I_{\alpha})\}_{\alpha\in \A\,\setminus\, \{ \pi_1^{-1}(1)\}},\ \quad \text{where}\ n_{J,\alpha}:=\min\{n\geq 1\mid T^{n}(\partial I_{\alpha})\in J\} \text{ for any } \alpha \in \A.
\]
If, in addition, $T$ has no connections, then the previous two sets have exactly $d - 1$ elements, and $T_J$ can be naturally seen as an IET on $d=\#\A$ intervals and identified with an element $(\pi_J, \lambda_J) \in S_0^\A \times \Lambda^{\A, J}$.

{The following simple auxiliary fact tells us what happens if $T$ has 
connections.}
\begin{lemma}\label{lem: alphabetafetrconnections}
Let $T$ be an ergodic interval exchange transformation of $d=\#\A$ intervals and let $J$ be a subinterval of the form \eqref{eq: paramofJ}. Then $T_J$ can be considered as an interval exchange of $d_J$ intervals, where
\[
d_J:= d - \#\{\alpha\in\A\,\setminus\,\{\pi_0^{-1}(1)\} \mid m_{J,\alpha}\ge M(\alpha)\}.
\]
In particular, if $J$ does not contain any point from any connection, then $d-d_J$ is equal to the number of non-trivial connections of $T$.
\end{lemma}
\begin{proof}
Since we know that the discontinuities of $T_J$ belong to $\{T^{-m_{J,\alpha}}(\partial I_{\alpha})\}_{\alpha\in \A\,\setminus\,\{\pi_0^{-1}(1)\}}$ which has at most $d-1$ elements, to prove that $T_J$ can be seen as an interval exchange of $d_J$ intervals it is sufficient to show that this set has exactly $d_J-1$ elements. 

Assume that $\alpha\in\A\,\setminus\,\{\pi_0^{-1}(1)\}$ is such that $m_{J,\alpha}\ge M(\alpha)$ and let $\beta\in \A\,\setminus\,\{\pi_0^{-1}(1)\}$ be such that $T^{-M(\alpha)}(\partial I_{\alpha})=\partial I_{\beta}$. Then, by the assumption on $\alpha$, we have that 
\[
T^{-m_{J,\alpha}}(\partial I_{\alpha})=T^{-m{J,\beta}}(\partial I_{\beta}).
\]
This shows that the connection that ends in the point $\partial I_{\alpha}$ decreases the number of discontinuities of $T_J$ by 1.
To conclude the proof of the first statement it remains to repeat the above reasoning for all $\alpha\in\A$ satisfying $m_{J,\alpha}\ge M(\alpha)$.

To prove the second assertion it is sufficient to notice that 
\[\#\{\alpha\in\A\,\setminus\,\{\pi_0^{-1}(1)\} \mid m_{J,\alpha}\ge M(\alpha)\} = \#\{\alpha\in\A\,\setminus\,\{\pi_0^{-1}(1)\} \mid M(\alpha) < +\infty\},\]
if $J$ does not contain any point from any connection.
\end{proof}

In view of the previous lemma, throughout this work, if $T: I \to I$ is an ergodic IET and $J \subseteq I$ is of the form \eqref{eq: paramofJ}, we will consider the induced IET $T_J$ as an IET on $d_J$ intervals, where
\[d_J = 1 + \# \{T^{-m_{J,\alpha}}(\partial I_{\alpha}) \mid \alpha\in \A\,\setminus\,\{\pi_0^{-1}(1)\} = d -\#\{\alpha\in\A\,\setminus\,\{\pi_0^{-1}(1)\} \mid m_{J,\alpha}\ge M(\alpha)\} \leq d.\]
We will also identify $T_J$ with an element $(\pi_J, \lambda_J) \in S_0^{\A_J} \times \Lambda^{\A_J, J}$ of a possibly smaller alphabet $\A_J$ 
and denote by $\{I_\gamma^J\}_{\gamma \in \A_J}$ the intervals exchanged by $T_J$.

Let us point out that, in the same way that an IET $T$ with $d$ intervals might have less than $d - 1$ discontinuities, the induced map $T_J$ might have less than $d_J - 1$ discontinuities, that is, some of the points in $
\{T^{-m_{J,\alpha}}(\partial I_{\alpha})\}_{\alpha\in \A\,\setminus\,\{\pi_0^{-1}(1)\}}$ might not be real discontinuity points of $T_J$.

Notice that given an ergodic IET $T: I \to I$ and a subinterval $J \subseteq I$ of the form \eqref{eq: paramofJ}, by the minimality of $T$ and $T^{-1}$, we can express $I$ as a disjoint union of the form
\begin{equation}
\label{eq:towers_decomposition}
I = \bigsqcup_{\gamma \in \A_J} \bigsqcup_{i = 0}^{h_\gamma - 1} T^i(I_\gamma^J),
\end{equation}
where, for any $\gamma \in \A_J$, $h_\gamma$ denotes the first return time to $J$ by $T$ of any point in $I_\gamma^J$.

 { 
 
 We finish this section by recalling a well-known fact, which we prove for completeness.

\begin{lemma}\label{lem: minimality}
 If $T: I \to I$ is ergodic with respect to the Lebesgue measure then it is minimal.
\end{lemma}
\begin{proof}
 We will show that for every $x,y\in I$ and every $\epsilon>0$ there exists $m\in\N$ such that $|T^{-m}y-x|<2\epsilon$. Take an interval $J:=[x,x+\epsilon)$ and consider the first return map $T_J$. Let $I^J_{\beta}$ be intervals exchanged by $T_J$ and $h_{\beta}$ the corresponding first return times. Since $T$ is ergodic, the set $\tilde I:=\bigcup_{\beta\in \mathcal B}\bigcup_{k=0}^{h_{\beta}-1} T^k(I^J_{\beta})$ is of full Lebesgue measure.
 
 Define $h:=\max\{h_{\beta}\mid\beta\in\mathcal B\}$. Consider the set $C:=\{T^{j}(\partial I_{\alpha})\mid \alpha\in\A,\ j=0,\ldots,h \}$. Pick $0<\delta<\epsilon$ such that $(y,y+\delta]\cap C=\emptyset$. Since $\tilde I$ is of full measure, there exists $\tilde y\in[y,y+\delta]\cap \tilde I$. By the choice of $\delta$ and by the fact that $T$ is right-continuous, the sets $\{T^{-j}[y,\tilde y]\mid j=0,\ldots,h \}$ are a family of pairwise disjoint intervals and $T^{-1}$ acts on each of them by translation. Since $\tilde y\in\tilde I$, there exists $m\le h$ such that $T^{-m}\tilde y\in J$ and $T^{-m}[y,\tilde y]=[T^{m}y,T^{-m}\tilde y]$. Thus 
 \[
|T^{-m}y,x|<\epsilon+\delta<2\epsilon,
 \]
 which finishes the proof.
 \end{proof}
}

\subsection{Parametrizing IETs with similar induced maps}
For every IET $T: I \to I$ given by $\pi \in S_0^\A$ and $\lambda\in\Lambda^{\A,I}$, we consider the set $\Lambda^{\mathcal A,I}_T$ given by
\begin{equation}
\label{eq:nbhd_connections}
\left\{ \tilde\lambda\in\Lambda^{\A,I} \left| 
\begin{array}{l}
M(T_{\tilde{\lambda}},\beta)=M(T,\beta)\ \text{and}\ N(T_{\tilde{\lambda}},\beta)=N(T,\beta), \quad \text{for}\ \beta\in\A, \\
T_{\tilde{\lambda}}^{-M(\beta)}(\partial I_\beta^{\tilde\lambda})=\partial I_\gamma^{\tilde\lambda}\Leftrightarrow T^{-M(\beta)}(\partial I_\beta)=\partial I_\gamma, \quad \text{for $\beta \in \A$ with }M(\beta)<\infty,\\
T_{\tilde{\lambda}}^{N(\beta)}(\partial I_\beta^{\tilde\lambda})=\partial I_\gamma^{\tilde\lambda}\Leftrightarrow T^{N(\beta)}(\partial I_\beta)=\partial I_\gamma, \quad \text{for $\beta \in \A$ with }N(\beta)<\infty.
\end{array}\right\} \right.
\end{equation}

In the above definition, the conditions on $N(\cdot)$ and $M(\cdot)$ are equivalent, nevertheless we write both of them for completeness. {The set $\Lambda^{\mathcal A,I}_T$ denotes all length vectors $\tilde \lambda$ in $\Lambda^{\A,I}$ for which the IET $(\pi, \tilde \lambda)$ has the same connection pattern as $T$. Obviously $\lambda\in \Lambda^{\mathcal A,I}_T$. The following proposition is one of the crucial tools used later in the proofs of the main results. In loose words, it states that by starting with any IET and considering a Rokhlin tower configuration obtained by inducing on a properly chosen interval, we can obtain a new IET by perturbing the parameters of this configuration, which has the same combinatorial and connection data as the initial map.

\begin{proposition}
\label{prop: unwinding}
Let $T = (\pi, \lambda) \in S_0^\A \times \Lambda^{\A, I}$ be an ergodic IET and let $J \subseteq I$ be a subinterval of the form \eqref{eq: paramofJ} with endpoints $T^{m}(\partial I_{\alpha_J}),$ $T^n(\partial I_{\beta_J}),$ for some $\alpha_J, \beta_J \in \A$ and $m, n \in \Z$. Assume that $J$ does not contain any point from any connection of $T$ and let
\begin{equation}
\label{eq:initial_towers}
I = \bigsqcup_{\gamma \in \A_J} \bigsqcup_{i = 0}^{h_\gamma - 1} T^i(I_\gamma^J),
\end{equation}
 be the associated Rokhlin tower decomposition of $I$. 
 
Then, for any $v \in \R^{\A_J}_+$ satisfying $\sum_{\gamma \in \A_J} v_\gamma h_\gamma = |I|$, there exists $\tilde \lambda \in \Lambda^{\A, I}$ such that the IET $\tilde T = (\pi, \tilde \lambda)$ and the interval $\tilde J$ with endpoints $\tilde T^{m}(\partial \tilde I_{\alpha_J})$, $\tilde T^n(\partial \tilde I_{\beta_J})$, where $\{ \tilde I_\alpha\}_{\alpha \in \A}$ denote the intervals exchanged by $\tilde T$, satisfy the following.
\begin{itemize}
\item The induced IET $\tilde T_{\tilde J}= (\tilde \pi ^{\tilde J}, \tilde \lambda ^{\tilde J})$ is defined on the alphabet $\A_J$,
\item $\tilde \lambda^{\tilde J} = v$ and $\tilde \pi^{\tilde J} = \pi^J$,
\item $\tilde T_{\tilde J}$ has the same associated tower decomposition as $T_J = (\pi^J, \lambda^J)$. 
\end{itemize}
\end{proposition}

In the following, given two intervals $J_1, J_2 \subseteq \R$, we denote 
\begin{equation}\label{eq: defofintervalorder}
J_1<J_2\ \Leftrightarrow\ x<y \text{ for any }x\in J_1\text{ and any }y\in J_2.
\end{equation}
Notice that given a collection of disjoint intervals, we can order it according to the relation above. 

\begin{proof}[Proof of Proposition \ref{prop: unwinding}]
Fix $v \in \R^{\A_J}_+$ satisfying $\sum_{\gamma \in \A_J} v_\gamma h_\gamma = |I|$. We will define the desired IET $\tilde T$ as follows. First, we will change the lengths of the intervals in the Rokhlin tower decomposition of $I$ associated with $T_J$, while keeping their order in $I$, to express $I$ as a disjoint union of intervals whose lengths are given by $v$. Then we will define a transformation $\tilde T$ on this union so that it defines a Rokhlin tower decomposition for the new transformation. Finally, we will check that $\tilde T$ Is an IET with the desired properties. 

Since $\sum_{\gamma \in \A_J} v_\gamma h_\gamma = |I|$, by changing the intervals of the form $T^i(I_\gamma^J)$ by intervals $\tilde I_{\gamma, i}$ of length $v_\gamma$ for every $\gamma \in \A_J$ and every $0 \leq i < h_\gamma$, we can express $I$ as a disjoint union of the form 
 \begin{equation}
\label{eq:modified_towers}
I = \bigsqcup_{\gamma \in \A_J} \bigsqcup_{i = 0}^{h_\gamma - 1} \tilde I_{\gamma, i}^J,
\end{equation}
 where 
 \[ \tilde I_{i, \gamma}^J < \tilde I_{j, \beta}^J \Leftrightarrow T^i(I_\gamma^J) < T^j(I_\beta^J), \]
 that is, the intervals in the decompositions \eqref{eq:initial_towers} and \eqref{eq:modified_towers} are ordered in the same way. 
 
Let \[\tilde J = \bigsqcup_{\gamma \in \A_J} \tilde I_{\gamma, 0}^J.\]
Notice that since our construction preserves the order of the intervals and $J = \bigsqcup_{\gamma \in \A_J} I_{\gamma}^J$ then $\tilde J$ is also an interval. Clearly $|J| = \sum_{\gamma \in \A_J} |\tilde I_{\gamma, 0}^J| = \sum_{\gamma \in \A_J} v_\gamma = |v|_1$. Moreover, since $J = \bigsqcup_{\gamma \in \A_J} T^{h_\gamma}(I_{\gamma}^J)$ we can also express $\tilde J$ as a disjoint union of intervals $\{ L^J_{\gamma} \}_{\gamma \in \A_J}$ of lengths given by $v$ and such that $\{ L^J_{\gamma} \}_{\gamma \in \A_J}$ and $\{ T^{h_\gamma}(I_{\gamma}^J) \}_{\gamma \in \A_J}$ have the same order inside $I$.

We define a transformation $\tilde T$ on $I$ by setting, for any $\gamma \in \A$, 
\begin{itemize}
 \item $\tilde{T}(\tilde I^J_{\gamma,i})=\tilde I^J_{\gamma,i+1}$ for $0\le i<h_{\gamma}-1$, \item $\tilde T(I^J_{\gamma, h_\gamma - 1}) = L_\gamma^J$,
\end{itemize}
and requiring $\tilde T$ to act via a translation when restricted to these subintervals.

Notice that with these definitions the images, by $T$ and $\tilde T$ respectively, of the intervals in the decompositions \eqref{eq:initial_towers} and \eqref{eq:modified_towers} are ordered in the same way, that is,
 \[ \tilde T(\tilde I_{i, \gamma}^J) < \tilde T(\tilde I_{j, \beta}^J ) \Leftrightarrow T \big( T^i(I_\gamma^J)\big) < T \big( T^j(I_\beta^J) \big), \]

We will show that $\tilde T$ can be seen as an IET with the same combinatorial data as $\pi$ and that its length vector belongs to $\Lambda^{\A, I}_T$. 

Denote $H:= \sum_{\gamma\in\mathcal \A_J}h_{\gamma}.$ Let $\{I_k\}_{k = 0}^{H - 1}$ be the intervals in the tower decomposition \eqref{eq:initial_towers} ordered according to their order in $I$, i.e.,
\[
I_{k_1}=T^{j_1}\big(I_{\gamma_1}^J\big)\text{ and }I_{k_2}=T^{j_2}\big(I_{\gamma_2}^J\big)\text{ with }k_1<k_2\ \Leftrightarrow\ T^{j_1}\big(I_{\gamma_1}^J\big)<T^{j_2}\big(I_{\gamma_2}^J\big),
\]
and let $\{I_k^+\}_{k = 0}^{H - 1}$ be the intervals in the same tower decomposition ordered according to the order of their images, i.e.,
\[
I_{k_1}^+=T^{j_1}\big(I_{\gamma_1}^J\big)\text{ and }I_{k_2}^+=T^{j_2}\big(I_{\gamma_2}^J\big)\text{ with }k_1<k_2\ \Leftrightarrow\ T\left(T^{j_1}\big(I_{\gamma_1}^J\big)\right)<T\left(T^{j_2}\big(I_{\gamma_2}^J\big)\right).
\]

Similarly, let $\{\tilde I_k\}_{k = 0}^{H - 1}$ and $\{\tilde I_k^+\}_{k = 0}^{H - 1}$ denote the intervals in the tower decomposition \eqref{eq:modified_towers} ordered according to their order in $I$ and the order of their images by $\tilde T$ in $I$, respectively.
 
For any $\alpha \in \A$ let $0\le k_\alpha< k_{\alpha}' < H$ and $0\le l_\alpha< l_{\alpha}' < H$ be such that
 \[
 I_{k}\subseteq I_{\alpha}\ \Leftrightarrow\ k_{\alpha}\le k\le k_{\alpha}', \qquad I_{\ell}^+\subseteq T(I_{\alpha})\ \Leftrightarrow\ \ell_{\alpha}\le \ell\le \ell_{\alpha}'.
 \]
Then for every $k_{\alpha}\le k\le k_{\alpha}'$, the IET $T$ acts as a translation on $I_{k}$ by $\sum_{j<\ell_\alpha}|I_j^+|-\sum_{j<k_\alpha}|I_j|$. 

We claim that for any $\alpha \in \A$, $\tilde I_{\alpha}:=\bigsqcup_{j=k_\alpha}^{k_\alpha'}\tilde I_j$ is an interval exchanged by $\tilde T$.

Indeed, since the intervals $\tilde I_k$ and $\tilde I_k^+$ are ordered in the same way as the intervals $I_k$ and $I_k^+$, respectively, the transformation $\tilde{T}$ acts on $\tilde I_k$ via translation by $\sum_{j<\ell_\alpha}|\tilde I_j^+|-\sum_{j<k_\alpha}|\tilde I_j|$. Since the translation value does not depend on $k$, $\tilde{T}$ acts as a translation on the whole interval $\tilde I_{\alpha}$. 

Therefore, we can see $\tilde T$ as an IET on $I$ with $\# \A$ intervals. Moreover, $\tilde T$ has the same combinatorics as $T$ since the intervals (and their images) in both tower decompositions are ordered in the same way. Thus we can identify $\tilde T$ with $(\pi, \tilde \lambda)$ for some $\tilde \lambda \in \Lambda^{\A, I}$, and we denote by $\{ \tilde I_\alpha\}_{\alpha \in \A}$ the intervals exchanged by $\tilde T$. 

Notice that the tower structure associated with $\tilde T_{\tilde J} = (\tilde \pi^{\tilde J}, \tilde \lambda^{\tilde J})$ is given by \eqref{eq:modified_towers}, that is, if we denote by $\{ \tilde I_\gamma^{\tilde J}\}_{\gamma \in \A_J}$ the intervals exchanged by $\tilde T_{\tilde J}$ then $ \tilde I_\gamma^{\tilde J} = \tilde I_{\gamma, 0}$, for every $\gamma \in \A_J$, and we have
\[ I = \bigsqcup_{\gamma \in \A_J} \bigsqcup_{i = 0}^{h_\gamma - 1} \tilde I_{\gamma, i}^J = \bigsqcup_{\gamma \in \A_J} \bigsqcup_{i = 0}^{h_\gamma - 1} \tilde T(\tilde I_{\gamma}^{\tilde J}).\]
Since the intervals $\{\tilde T_{\tilde J}(\tilde I_\gamma^{\tilde J})\}_{\gamma \in \A_J} = \{\tilde T^{h_\gamma}(\tilde I_\gamma^{\tilde J})\}_{\gamma \in \A_J} = \{L_\gamma^J\}_{\gamma \in \A}$ have the same order as the intervals $\{T_{J}(I_\gamma^J)\}_{\gamma \in \A_J} = \{T^{h_\gamma}(I_\gamma^J)\}_{\gamma \in \A_J}$ it follows that $\tilde \pi_{\tilde J} = \pi^J$. Finally, since the lengths of the intervals $\{\tilde I_\gamma^{\tilde J}\}_{\gamma \in \A_J}$ are given by $v$, it follows that $\tilde \lambda^{\tilde J} = v$.

Moreover, notice that the endpoints of the intervals exchanged by $\tilde T$ and those exchanged by $T$ belong to the same tower floors in their respective decompositions, that is,
\[ \partial \tilde I_\alpha \in \tilde{T}^i(\tilde I_\gamma^{\tilde J}) \Leftrightarrow \tilde I_\alpha \in T^i(I^J_\gamma),\]
 for any $\gamma \in \A_J$ and any $0 \leq i < h_\gamma$. From this, and since $J$ verifies \eqref{eq: paramofJ}, it follows that $\tilde J$ is an interval with endpoints $\tilde T^{m}(\partial \tilde I_{\alpha_J})$, $\tilde T^n(\partial \tilde I_{\beta_J})$.
 
Since $J$ does not contain any point from any connection, then every connection is contained in one of the towers of the form $\bigsqcup_{i=0}^{h_{\gamma}-1}T^i(I_{\gamma}^J)$, for some $\gamma \in \A_J$. Then, it follows from the definition of $\tilde T$ (and the previous remarks concerning the endpoints of $\tilde T$) that $\tilde T$ possesses the same connection pattern as $T$, that is, $\tilde \lambda \in \Lambda_T^{\A, I}$. 
\end{proof}

{
\begin{remark}
\label{rmk:simplex}
The previous proposition defines a $(d - d' - 1)$-dimensional simplex $\Delta_J \subseteq \Lambda^{\A, I}$ around $\lambda$ such that for any $\tilde \lambda \in \Delta_J$ the IET $(\pi, \tilde\lambda)$ verifies the conclusions of the proposition, where $d'$ denotes the number of non-trivial connections of $(\pi, \lambda)$.

Indeed, It follows from the proof that the map $v \mapsto \tilde{\lambda}(v)$ given by the previous proposition is linear on the simplex $\tilde \Delta_J:= \big\{ v \in \R^{\A_J}_+ \, \big| \, \sum_{\gamma \in \A_J} v_\gamma h_\gamma = |I| \big\} $ and that $\tilde\lambda(\lambda^J) = \lambda$. Moreover, the map is also injective since we can recover $v$ by inducing the IET associated to $(\pi, \tilde\lambda(v))$ to the interval $\tilde J$. Notice that, by Lemma \ref{lem: alphabetafetrconnections}, the simplex $\tilde \Delta_J$ has dimension $d - d' -1$. Thus $\Delta_J = \tilde\lambda(\tilde \Delta_J)$ satisfies the statement above.

\end{remark}
}

\section{Symmetric interval exchange transformations}

\subsection{Notations and basic properties} 
{Let $I = [a, b)$ be a bounded interval and $T=(\pi, \lambda) \in S_0^\A \times \Lambda^{\A, I}$ be an IET on $I$ with $d := \#\A$ intervals. The permutation $\pi$ (and any IET having $\pi$ as permutation) is said to be} \textit{symmetric} if $\pi_{1}\circ \pi^{-1}_{0}(i)=d+1-i$, {for any $1 \leq i \leq d$.} We say that $T$ is \emph{non-degenerate} if $\partial I_{\alpha}$ is a discontinuity of $T$, for every $\alpha\in \A {\,\setminus\, \{\pi_0^{-1}(1)\}}$. {If $T$ is \emph{degenerate}, i.e., if there exists $\partial I_{\alpha}$ which is not a real discontinuity, then whenever we refer to \emph{intervals of continuity} of $T$, we mean maximal intervals of continuity. Notice that the inverse of a symmetric IET is also a symmetric IET.


We denote the \emph{symmetric reflection} or \emph{involution} {on the open interval $\mathring{I} = (a, b)$ by $\mathcal{I}_I$, where $\mathcal{I}_I: \mathring{I}\to \mathring{I}$ is given by $\mathcal{I}_I(x)= a + b - x$. We omit the endpoints of the interval in this definition so that the domain and codomain of the involution are well-defined subsets of $I$.}
{It is well-known and easy to verify that if $T$ is a symmetric IET then
\begin{equation}\label{eq: standardconjugacy}
\mathcal{I}_I\circ T(x)=T^{-1}\circ\mathcal{I}_I(x), \quad \text{ if }\quad x \neq \partial I_{\alpha}, \text{ for any } \alpha \in \A.
\end{equation}
}
{More generally, the equation above implies 
\begin{equation}
\label{eq:generalized_conjugacy}
\mathcal{I}_I\circ T^n(x)=T^{-n}\circ\mathcal{I}_I(x), \quad \text{ if } \quad x \neq T^{-i}(\partial I_\alpha), \text{ for any } \alpha \in \A \text{ and any } 0 \leq i < n.
\end{equation}}
 {Notice that $\I_I\circ T$ and $T^{-1}\circ \I_I$ are not defined everywhere on $I$ since the $\I_I\circ T$ is not defined at $\partial I_{\pi_0^{-1}(d)}$ while $T^{-1}\circ \I_I$ is not defined at $\partial I_{\pi_0^{-1}(1)}$.} 
Moreover, {a direct calculation shows that}
\begin{equation}
\label{eq:sym_identity_endpoints}
\I_I\circ T(\partial I_{\alpha})=\partial I_{\hat\alpha},\ \text{where}\ \pi_0(\hat\alpha)=\pi_0(\alpha)+1,
\end{equation}
for $\alpha\in\A$ with $\pi_0(\alpha)\neq d$, and 
\begin{equation}\label{eq:sym_identity_endpoints_2}
T^{-1}\circ \I_I(\partial I_{\alpha})=\partial I_{\hat\alpha},\ \text{where}\ \pi_0(\hat\alpha)=\pi_0(\alpha)-1,
\end{equation}
for $\alpha\in\A$ with $\pi_0(\alpha)\neq 1$.

 Let us point out that $\pi$ being symmetric is not a necessary condition for \eqref{eq: standardconjugacy} to hold. Indeed, if $\pi$ is not symmetric but its intervals of continuity are exchanged symmetrically {(e.g., by adding a `fake discontinuity' in one of the exchanged intervals of a symmetric IET with $d$ intervals and considering it as an IET on $d + 1$ intervals)} then \eqref{eq: standardconjugacy} still holds. Moreover, there are examples of IETs that do not exchange their intervals of continuity symmetrically but still satisfy \eqref{eq: standardconjugacy}. These examples arise from two-covers of quadratic differentials {(see, e.g., \cite{boissy_dynamics_2009}}).

The following lemma provides a simple sufficient condition for an IET satisfying \eqref{eq: standardconjugacy} to be symmetric.

\begin{lemma}
\label{lem: condition for symmetric iets,}
 Let $T=(\pi,\lambda): I \to I$ be a non-degenerate IET. If $T$ satisfies
 \eqref{eq: standardconjugacy}
 and $\lambda_{\alpha}\neq \lambda_{\beta}$ for any distinct $\alpha, \beta \in \mathcal{A}$, then $T$ is symmetric.
\end{lemma}
\begin{proof}
 Arrange the intervals $\{I_{\alpha}\}_{ \alpha\in \A}$ according to their lengths, such that $\lambda_{\alpha_1}>\lambda_{\alpha_2}>\dots> \lambda_{\alpha_d}$.
 Note that $\I\circ T$ is continuous on $\mathring I_{\alpha_1}$, so $T^{-1}\circ \I$ is also continuous on $\mathring I_{\alpha_1}$. Since the discontinuity points of $T^{-1}$ are given by $\{T(\partial I_{\beta}) \mid \beta\in \A {\setminus \{\pi_1^{-1}(1)\}} \}$, so by the maximal length of $I_{\alpha_1}$, non-degenericity of $T$ and the continuity of $T^{-1}\circ \I$, we must have 
 
 \begin{equation}\label{symmetric identity}
 \I(I_{\alpha_1})=T(I_{\alpha_1}). 
 \end{equation}

 By induction on the index $\{\alpha_i, 1\le i\le d\}$, the above identity holds for every $\alpha \in \A$. Because the involution $\I$ reverses the order of $\{I_{\alpha}, \alpha \in \A\}$, that is:
 \[
 I_{\alpha}<I_{\beta}\Longrightarrow \I(I_{\alpha})>\I(I_{\beta}),
 \]
 {where the order of intervals is given by \eqref{eq: defofintervalorder}, }the identity \eqref{symmetric identity} implies that $T$ also reverses the order of $\{I_{\alpha}, \alpha \in \A\}$, hence $T$ is symmetric.
\end{proof}
{We have an immediate consequence for IETs with general combinatorial data.
\begin{corollary}\label{cor: symmetricintcont}
 Let $T: I \to I$ be an IET satisfying \eqref{eq: standardconjugacy} such that all its continuity intervals are of different lengths. Then $T$ exchanges its continuity intervals symmetrically.
\end{corollary}
\begin{proof}
 The result follows from Lemma \ref{lem: condition for symmetric iets,} by replacing the intervals exchanged by $T$ with the continuity intervals of $T$ (thus possibly reducing the number of exchanged intervals).
\end{proof}
}

{Given an IET $T: I \to I$ with exchanged intervals $\{I_{\alpha}\}_{\alpha \in \A}$ we denote by $c_{\alpha}$ the middle point of the interval $I_\alpha$, for any $\alpha\in\A$. Note that if $T$ is symmetric, then 
\begin{equation}\label{eq: centersandinvolution}
 \mathcal{I}_I\circ T(c_{\alpha})=c_{\alpha},
\end{equation}
 for every $\alpha\in\A$. These points, as well as the middle point of the interval $I$, which we denote by
 \[c_{1/2} := \tfrac{a + b}{2},\] will play an important role in our proofs. Notice that $c_{1/2}$ is the only fixed point of $\I_I$ while 
 {the points $\{c_{\alpha}\}_{\alpha\in\A}$ are the only ones satisfying 
 \eqref{eq: centersandinvolution}.} {Moreover, the backward and forward 
 iterates of these points are closely related.
 
 \begin{lemma}
 \label{lem:inverse_iterates}
 Let $T: I \to I$ be a symmetric IET and $\alpha \in \A \cup \{\tfrac{1}{2}\}$. If $m \geq 1$ is such that $\{c_\alpha, T(c_\alpha), \ldots, T^{m - 1}(c_\alpha)\} \cap \{\partial I_\beta\}_{\beta \in \A} = \emptyset$, then
 \begin{equation}
 \label{eq:inverse_iterates}
 T^{-m}(c_\alpha) = 
 {T^{-1} \circ \I_I} \big( T^{m - \delta_{1/2}(\alpha)}(c_\alpha) \big),
 \end{equation}
 where 
 \begin{equation}
\label{eq:delta_1/2}
\delta_{1/2}(\beta) = \left\{ \begin{array}{ll} 1 & \text{ if } \quad \beta = \tfrac{1}{2}, \\
0 & \text{ otherwise}.\\
 \end{array} \right.
 \end{equation}
 \end{lemma}
}

\begin{proof}
This result follows by directly applying \eqref{eq:generalized_conjugacy} and noticing that $\I_I(c_{1/2}) = c_{1/2}$ and that $\I_I(c_\alpha) = T(c_\alpha)$, for any $\alpha \in \A$.
\end{proof}

{The following result concerning Birkhoff sums over symmetric IETs {follows 
directly} from \cite[Lemma 3.11]{berk_ergodicity_2023}. Although not 
immediately useful for us, this fact will be crucial in one of the final 
arguments of the proof of the main result of this paper.
 \begin{lemma}\label{lem: berktrujillo}
 Let $T: I \to I$ be a symmetric IET, $\alpha \in \A \cup\{\tfrac{1}{2}\}$ and $f : I\to \R$ satisfying $f\circ \I_I=-f$ on $\mathring I$. If $c_{\alpha}$ is not part of any connection, then $$S_{2n + \delta_{1/2}(\alpha)}f(T^{-n}(c_{\alpha})) = 0,$$ for any $n\in\N$, where $\delta_{1/2}$ is given by \eqref{eq:delta_1/2}.
 \end{lemma}
 }

\subsection{Induced symmetric IETs} 

 Given an IET $T: I \to I$ and a subinterval $J \subseteq I$ with associated induced map $T_J = (\pi^J, \lambda^J) \in S_0^{\A_J} \times \Lambda^{\A_J, J}$ (see Section \ref{sc:induced_IETs} and Lemma \ref{lem: alphabetafetrconnections}), we denote by $\{I^J_\gamma\}_{\gamma \in \A_J}$ the intervals exchanged intervals by $T_J$ and by $\{c^J_\gamma\}_{\gamma \in \A_J}$ the middle points of these intervals. We denote by $p_J : I \to J$ the first return map to $J$ by $T^{-1}$, that is, $p_J$ is given by $x \mapsto T^{-b_J(x)}(x)$ where, 
 \begin{equation}
 \label{eq:backward_return}
 b_J(x) := \min\{m \geq 1 \mid T^{-m}(x) \in J\}.
 \end{equation}
We say that a subinterval $J \subseteq I$ is \emph{symmetric} if there exists $\alpha \in \A$ and $\Delta > 0$ such that $J = [c_\alpha - \Delta, c_\alpha + \Delta) \subseteq I_\alpha$ {or $J=\big[c_{1/2}-\Delta, c_{1/2} +\Delta\big)\subseteq I$. To differentiate between the two cases we will refer to the former as \emph{$\alpha$-symmetric} and to the latter as \emph{$\tfrac{1}{2}$-symmetric}.}


As we shall see, inducing a symmetric IET on symmetric intervals defines IETs satisfying \eqref{eq: standardconjugacy}, which, as seen in the previous section, is closely related with the symmetricity of IETs (see Lemma \ref{lem:symmetric_condition}). Moreover, it is possible to construct $\alpha$-symmetric subintervals of the form \eqref{eq: paramofJ}, for any $\alpha \in \A \cup \{\tfrac{1}{2}\}$ (see Lemma \ref{lem: sym_endpoints}). Under additional conditions on the symmetric subinterval, we can guarantee that associated induced maps are actually symmetric. 

\begin{proposition}
\label{prop:induced_symmetric}
Let $T: I \to I$ be an ergodic symmetric IET and fix $\alpha \in \A \,\setminus\, \{\pi_0^{-1}(1)\}$ such that $M(\alpha) = +\infty$ (see Corollary \ref{cor: existenceofinfiniteorbit}). Fix $m \geq 1$ and let $J \subseteq I$ be the left-closed right-open subinterval with endpoints $T^{-m}(\partial I_\alpha)$, $T^{m}(\partial I_{\hat \alpha})$ (resp. $T^{-m + 1}(\partial I_\alpha),$ $T^{m}(\partial I_{\hat \alpha})$), where $\pi_0(\hat \alpha) = \pi_0(\alpha) - 1$. 
Then $J$ is $\beta$-symmetric for some $\beta \in \A$ (resp. $\tfrac{1}{2}$-symmetric) and $T_J$ satisfies \eqref{eq: standardconjugacy}. 

If, in addition, $J$ does not contain points from any connection. Then $T_J$ is an ergodic symmetric IET on $d - d'$ intervals, where $d'$ is the number of non-trivial connections of $T$.


\end{proposition}

In the setting of the previous proposition, the exchanged intervals' middle points of the induced map and of the original IET will be closely related.

\begin{proposition}
\label{prop:induced_centers}
Let $T$, $\alpha, \hat \alpha, \beta$ and $J$ as in Proposition \ref{prop:induced_symmetric}. Assume that $J$ does not contain points from any connection. Then the following holds:
\begin{enumerate}
\item For any $\gamma \in \A_J$ there exists unique $\sigma \in \A \cup \{\tfrac{1}{2}\}$, with $ \sigma \neq \alpha$, and $\ell \geq 1$ such that $c_\gamma^J = p_J(c_\sigma) = T^{-\ell}(c_\sigma)$ and $T_J(p_J(c_\sigma)) = T^{\ell - \delta_{1/2}(\sigma)}(c_\sigma)$, where $\delta_{1/2}$ is given by \eqref{eq:delta_1/2}. Moreover, $c_\sigma$ does not belong to any non-trivial connection. 
\item For any $\delta \in \A \cup \{\tfrac{1}{2}\}$ with $\delta \neq \alpha$, either $p_J(c_\delta) = c_\gamma^J$ or $p_J(c_\delta) = \partial I_\gamma^J$ for some $\gamma \in \A_J$. In the latter case, $c_\delta$ lies inside a connection of $T$.
\end{enumerate}
\end{proposition}

For the sake of clarity, we postpone the proofs of Propositions \ref{prop:induced_symmetric} and \ref{prop:induced_centers} to the end of this section and rather start by proving several preliminary lemmas.

Given a symmetric IET $T: I \to I$, it is easy to verify that the interior of $\tfrac{1}{2}$-symmetric intervals are invariant by $\I_I$. On the other hand, as we shall see below, the interior of $\alpha$-symmetric intervals are invariant by $\I_I \circ T$. 
}

\begin{lemma}\label{lem: local_reflection}
	{Let $T : I \to I$ be a symmetric IET and $J \subseteq I$ be an $\alpha$-symmetric interval, for some $\alpha \in \A$.} Then $\mathcal{I}_I\circ T(\mathring{J}) = \mathring{J}$. Moreover, $$\I_J = \mathcal{I}_I\circ T|_{\mathring{J}},$$ that is, $\mathcal{I}_I\circ T|_{\mathring{J}}$ is the symmetric reflection on $J$.
\end{lemma}
\begin{proof}
	{Let $\alpha \in \A$ and $\Delta > 0$ such that $J = [c_\alpha - \Delta, c_\alpha + \Delta) \subseteq I_\alpha$.} Fix $x\in J$ and let $-\Delta < \delta < \Delta$ be such that $x = c_{\alpha}+\delta$. {Notice that $\I_J(x) = c_\alpha - \delta$.} 
 
	Since $T|_{I_{\alpha}}$ is a translation, by \eqref{eq: centersandinvolution}, we have
	\[
	\I_I\circ T(x)=\I_I\circ T(c_{\alpha}+\delta)=\I_I\left(T(c_{\alpha}) +\delta\right)=
	\I_I\circ T(c_{\alpha})-\delta=c_{\alpha}-\delta\,\in\, J,
		\]
		which finishes the proof.
\end{proof}

{Notice that any of the exchanged intervals of a symmetric IET $T: I \to I$ defines a symmetric interval. Hence, we obtain the following as a direct consequence of Lemmas \ref{lem:inverse_iterates} and \ref{lem: local_reflection}.
\begin{corollary}
Let $T : I \to I$ be a symmetric IET, $\alpha \in \A$ and $m \geq 1$. Then $\{c_\alpha, \ldots, T^{m - 1}(c_\alpha)\} \cap \{\partial I_\beta\}_{\beta \in \A} = \emptyset$ if and only if $\{c_\alpha, \ldots, T^{-m + 1}(c_\alpha)\} \cap \{\partial I_\beta\}_{\beta \in \A} = \emptyset$. 
\end{corollary}

We now relate the induced IET on a symmetric interval $J$ with the associated involution $\I_J$. 

\begin{lemma}
\label{lem:symmetric_condition}
	{Let $T : I \to I$ be a symmetric IET and $J \subseteq I$ be a symmetric interval. Then $T_J$ satisfies \eqref{eq: standardconjugacy}.}
\end{lemma}
\begin{proof}
	{Assume first that $J$ is $\alpha$-symmetric for some $\alpha\in\A$}. Let $x\in J$ {be not an endpoint of an interval exchanged by $T_J$} and let $h\in\N$ be such that $T_J(x)=T^{h}(x)$. {Notice that since $x$ is not an endpoint of the exchanged intervals, }{by \eqref{eq:generalized_conjugacy} applied to $T$,}
	\[
{(\I\circ T)\circ T_{J}(x)} = (\I\circ T) \circ T^h(x)=T^{-h}\circ (\I\circ T)(x).
	\]
 {Thus, to prove \eqref{eq: standardconjugacy}, it suffices to notice that $$T_{J}^{-1}\circ (\I\circ T)(x) = T^{-h}\circ (\I\circ T)(x).$$ 
 Indeed, by Lemma \ref{lem: local_reflection}, for any $n \geq 1$, $T^n(x) \in J$ if and only if $(\I\circ T) \circ T^n(x) \in J$. Thus the equation above follows since $T_J(x)=T^{h}(x)$ and $T^{-h}\circ (\I\circ T)(x) = (\I\circ T) \circ T^h(x).$}

 {Assume now that $J$ is $\tfrac{1}{2}$-symmetric. Let $x\in J$ be not an endpoint of an interval exchanged by $T_J$ and let $h\in\N$ be such that $T_J(x)=T^{h}(x)$. Again by \eqref{eq:generalized_conjugacy}} we get
 	\[
 \I\circ T_{J}(x) = \I\circ T^h(x)=T^{-h}\circ \I(x)
	\]
 and, similarly to the previous case, to show \eqref{eq: standardconjugacy} it is sufficient to notice that $T_{J}^{-1}\circ\I(x)=T^{-h}\circ\I(x)$.
\end{proof}

By the result above and given Corollary \ref{cor: symmetricintcont}, if $J$ is symmetric and the lengths of the continuity intervals of $T_J$ are pairwise distinct, then $T_J$ exchanges its continuity intervals symmetrically. 

The following lemma shows that if the orbit of a middle point intersects the discontinuities of the IET, then the middle point must be part of a connection.

\begin{lemma}
\label{lem:symmetric_connection}
Let $T: I \to I$ be a symmetric IET and $\beta \in \A \cup \{\tfrac{1}{2}\}$. Assume there exists $m \in \Z$ and $\alpha \in \A$ such that $T^m(c_\beta) = \partial I_\alpha$ and $T^k(c_\beta) \notin \{ \partial I_\gamma\}_{\gamma \in \A}$, for any $-|m| < k < |m|$.

Then {if $m\ge 0$}, we have $\alpha \neq \pi_0^{-1}(1)$ and
\[T^{-m - \delta_{1/2}(\beta)}(c_\beta) = \partial I_{\hat \alpha},\]
where $\delta_{1/2}$ is given by \eqref{eq:delta_1/2} and $\pi_0( \hat \alpha) = \pi_0(\alpha) - 1$.

{If on the other hand $m<0$, then $\alpha\neq \pi_1^{-1}(1)$ and
\[T^{-m - \delta_{1/2}(\beta)}(c_\beta) = \partial I_{\hat \alpha},\]
with $\pi_0( \hat \alpha) = \pi_0(\alpha) + 1$.
}

 In particular, $c_\beta$ lies inside a non-trivial connection. 
\end{lemma}
\begin{proof}

 If $m = 0$, then necessarily $\beta = \tfrac{1}{2}$ and we have $c_{1/2} = \partial I_\alpha$ with $\pi_1(\alpha) \neq 1$. By \eqref{eq:sym_identity_endpoints_2}, $T^{-1}(c_{1/2}) = \partial I_{\hat \alpha}$, where $\pi_0( \hat \alpha) = \pi_0(\alpha) - 1$.
 
Without loss of generality, let us assume $m < 0$, the case $m > 0$ being analogous. By \eqref{eq:inverse_iterates},
\begin{equation*}
\label{eq:inverse_iterates_negative}
T^m(c_{\beta})=T^{-1}\circ \I_I(T^{-m-\delta_{1/2}(\beta)})\ \Leftrightarrow\ T^{-m- \delta_{1/2}(\beta) }(c_\beta) = \I_I\circ T \big( T^{m}(c_\beta)\big).
\end{equation*}
Notice that since $T^{m + 1}(c_\beta) \notin \{ \partial I_\gamma\}_{\gamma \in \A}$ then $T^{m}(c_\beta) = \partial I_\alpha \neq \partial I_{\pi_0^{-1}(1)}$. Hence, by \eqref{eq:sym_identity_endpoints}, the equation above implies $T^{-m - \delta_{1/2}(\beta)}(c_\beta) = \partial I_{\hat \alpha}$, where $\pi_0( \hat \alpha) = \pi_0(\alpha) + 1$.


\end{proof}

The previous lemma immediately implies the following. 

\begin{corollary}
\label{cor:centers_connections}
Let $T: I \to I$ be a symmetric IET. Then any non-trivial connection of $T$ contains at most one point from the set $\left\{ c_\alpha \mid \alpha \in \A \cup \{\tfrac{1}{2}\} \right\}.$ In particular, there exists $\alpha \in \A$ such that $c_\alpha$ does not belong to any connection.
\end{corollary}

The following result provides a `recipe' to construct symmetric intervals dynamically by using iterates of the endpoints of the exchanged intervals.}


\begin{lemma}\label{lem: sym_endpoints}
{Let $T : I \to I$ be a symmetric IET.}
Let $\alpha\in\mathcal A\setminus\{\pi_0^{-1}(1)\}$ and 
{$m< M(\alpha)$}. Then 
\begin{equation}\label{eq: connections}
\I_I\circ T(T^{-m}(\partial I_{\alpha}))={\I_I\circ (T^{-m+1}(\partial I_{\alpha}))}=T^m(\partial I_{\hat\alpha}),
\end{equation}
where $\pi_0(\hat\alpha)=\pi_0(\alpha)-1$.

 In particular, the left-closed right-open interval with endpoints $T^{-m}(\partial I_{\alpha})$ and $T^m(\partial I_{\hat\alpha})$ is {$\beta$-symmetric for some $\beta\in\A$, while the left-closed right-open interval with endpoints $T^{-m+1}(\partial I_{\alpha})$ and $T^m(\partial I_{\hat\alpha})$ is $\tfrac{1}{2}$-symmetric}. 
 
 Moreover, $M(\alpha)= N(\hat\alpha)$.
\end{lemma}
\begin{proof}
{
First, notice that \eqref{eq: connections} is equivalent to \eqref{eq:sym_identity_endpoints} if $m = 1$. This implies that $\I_I(\partial I_\alpha) = T(\partial I_{\hat\alpha})$.} 
In the following, we assume $m > 1$. If $m < M(\alpha)$, by \eqref{eq:generalized_conjugacy}, 
\[\I\circ T(T^{-m}(\partial I_{\alpha}))= \I \circ T^{-m + 1}(\partial I_\alpha) = T^{m - 1} \circ \I (\partial I_\alpha) = T^m(\partial I_{\hat \alpha}). \]
This proves \eqref{eq: connections}, which easily implies that the symmetricity properties of the intervals in the statement.

Assume now that $m=M(\alpha)$ and $T^{-m}(\partial I_{\alpha})=\partial I_{\beta}$ for some $\beta\in \A$. Since $M(\alpha)$ is the minimal number that makes a connection, by \eqref{eq: connections} we have 
\[
\I\circ T(T(\partial I_{\beta}))=\I\circ T(T^{-{m+1}}\partial I_{\alpha})=T^{m-1}(\partial I_{\hat\alpha}).
\]
Moreover, noticing that $T(\partial I_\beta) \notin \{ \partial I_\alpha\}_{\alpha \in \A}$ since $m > 1$, by \eqref{eq: standardconjugacy} 
\[
\I\circ T(T(\partial I_{\beta}))=T^{-1}\circ\I\circ T(\partial I_{\beta})
=T^{-1}(\partial I_{\hat\beta}),
\]
where $\pi_0(\hat\beta)=\pi_0(\beta)+1$. Thus we get $T^{m}(\partial I_{\hat\alpha})=\partial I_{\hat\beta}$. This proves $N(\hat\alpha)\le M(\alpha)$. The opposite inequality is proven analogously, {by exchanging the roles of $T$ and $T^{-1}$.}
\end{proof}

Using the previous lemma, it is not difficult to construct symmetric intervals as in the statement of Proposition \ref{prop:induced_symmetric}. 

\begin{lemma}
\label{lem:sym_int_disjoint}
Let $T: I \to I$ be an ergodic symmetric IET. Then there exists $\beta \in \A$ such that $c_\beta$ is not part of any non-trivial connection from $T$ and, for any $\epsilon > 0$, there exists a $\beta$-symmetric subinterval $J\subseteq I$ disjoint from the connections of $T$ satisfying $|J| < \epsilon$.
\end{lemma}
\begin{proof}
By Corollary \ref{cor:centers_connections}, there exists $\beta \in \A$ such that $c_\beta$ is not part of any non-trivial connection of $T$ and, by Corollary \ref{cor: existenceofinfiniteorbit}, there exists $\alpha \in \A \,\setminus\, \{\pi_0^{-1}(1)\}$ such that $M(\alpha) = +\infty$. 

Since $T$ is ergodic and hence minimal, there exists $m \geq 1$ such that $\big|T^{-m}(\partial I_\alpha) - c_\beta \big| < \tfrac{\epsilon}{2}.$ By taking $\epsilon$ smaller if necessary, we may assume that $\epsilon < \min_{\delta \in \A} |I_\delta|$ and that $(c_\beta - \epsilon, c_\beta + \epsilon)$ does not contain any point from any connection.

Then, by Lemma \ref{lem: sym_endpoints}, the left-closed right-open subinterval $J$ with endpoints $T^{-m}(\partial I_\alpha)$, $T^{m}(\partial I_{\hat \alpha})$ is $\beta$-symmetric, where $\pi_0(\hat \alpha) = \pi_0(\alpha) - 1$. In particular $J \subseteq (c_\beta - \tfrac{\epsilon}{2}, c_\beta + \tfrac{\epsilon}{2}).$
\end{proof}

We are now in a position to prove Propositions \ref{prop:induced_symmetric} and \ref{prop:induced_centers}.

\begin{proof}[Proof of Proposition \ref{prop:induced_symmetric}]
Let $T = (\pi, \lambda), J, m, \alpha, \hat \alpha$ as in the statement of the proposition, with $J$ not necessarily disjoint from the connections of $T$.

By Lemma \ref{lem: sym_endpoints}, there exists $\beta \in \A \cup \{\tfrac{1}{2}\}$ such that $J$ is $\beta$-symmetric, where $\beta = \tfrac{1}{2}$ only if the endpoints of $J$ are of the form $T^{-m + 1}(\partial I_\alpha)$, $T^m(\partial I_{\hat \alpha})$. Then, by Lemma \ref{lem:symmetric_condition}, the induced IET $T_J$ satisfies \eqref{eq: standardconjugacy}. 

From now on, let us assume that $J$ does not contain any point from any connection. 

By Lemma \ref{lem: alphabetafetrconnections}, it follows that $T_J = (\pi^J, \lambda^J)$ is an IET on $d - d'$ intervals, where $d'$ is the number of non-trivial connections of $T$. In view of Proposition \ref{prop: unwinding}, there exists $\tilde\lambda\in\R^{\A}$ such that $\tilde{T}:=(\pi,\tilde\lambda)$ is symmetric, the IET $\tilde{T}_{\tilde J} = (\tilde \pi^J, \tilde \lambda^J)$ obtained by inducing $\tilde{T}$ to the interval $\tilde J$ with endpoints $T^{-m}(\partial \tilde I_{\alpha})$ and $T^{m}(\partial \tilde I_{\hat\alpha})$ has the same combinatorics as $T_J$, and $\tilde\lambda^J$ has intervals of rationally independent lengths. 

Since $\tilde J$ is a symmetric interval for $\tilde T$, by Lemma \ref{lem:symmetric_condition} the induced IET $\tilde T_J$ satisfies \eqref{eq: standardconjugacy}. Thus, by Corollary \ref{cor: symmetricintcont}, $\tilde{T}_{\tilde J}$ exchanges its maximal continuity intervals symmetrically, and therefore so does $T_J$.

Hence, to prove that $T_J$ is a symmetric IET, it suffices to show that it possesses exactly $d - d'$ maximal continuity intervals. By replacing $T_J$ with $\tilde{T}_{\tilde J}$ if necessary, we may assume that the intervals exchanged by $T_{J}$ are of rationally independent lengths. Then also the intervals of continuity of $T_J$ are of rationally independent lengths. Let $\{I^J_{\alpha}\}_{\alpha\in\mathcal \A_J}$ be the intervals exchanged by $T_J$ and let $\{\hat I^J_{\alpha}\}_{\alpha\in\mathcal C}$ be the maximal continuity intervals of $T_J$. Then, to finish the proof, it suffices to show that $\#\mathcal C = \#\mathcal \A_J$. 
 
For every $\gamma\in\mathcal C$ consider the point $\hat c_\gamma^J$, the center-point of the interval $\hat I^J_{\gamma}$. Since $T_J$ interchanges the intervals of continuity symmetrically, we have $T_J(\hat c_\gamma^J)=\I_J(\hat c_\gamma^J)$ and no other point satisfies this equation. Moreover, since the lengths of intervals exchanged by $T_J$ are rationally independent, we also have $\hat c_\gamma^J \notin\{\partial I^J_{\alpha}\}_{\alpha\in\A_J}.$ 

\begin{claim}
\label{cl:backward_centers}
Let $\sigma \in \A \cup \{\tfrac{1}{2}\}$ with $\sigma \neq \beta$ and $\ell := b_J(c_\sigma) \geq 1$ be the first backwards return time of $c_\sigma$ to $J$, that is, such that $p_J(c_\sigma) = T^{-\ell}(c_\sigma)$. Then, the following dichotomy holds.
\begin{itemize}
\item either $\{c_\sigma, T(c_\sigma), \ldots, T^{\ell - \delta_{1/2}(\sigma)}(c_\sigma)\} \cap \{\partial I_\delta\}_{\delta \in \A} = \emptyset$ and $$p_J(c_\sigma) = \hat c_\gamma^J,\qquad T_J(\hat c_\gamma^J) = T^{\ell - \delta_{1/2}(\sigma)}(c_\sigma) = T^{2\ell - \delta_{1/2}(\sigma)}(\hat c_\gamma^J), \quad \text{ for some } \gamma \in \mathcal C,$$
\item or $\{c_\sigma, T(c_\sigma), \ldots, T^{\ell - \delta_{1/2}(\sigma))}(c_\sigma)\} \cap \{\partial I_\delta\}_{\delta \in \A} \neq \emptyset$ and $$p_J(c_\sigma) = \partial I_\gamma^J,\quad\text{ for some } \gamma \in \A_J,$$ 
\end{itemize}
where $\delta_{1/2}$ is given by \eqref{eq:delta_1/2}. Moreover, in the latter case, $c_\sigma$ lies in a non-trivial connection of $T$.
\end{claim}

\begin{proof}
First, let us assume that $\{c_\sigma, T(c_\sigma), \ldots, T^{\ell - \delta_{1/2}(\sigma)}(c_\sigma)\} \cap \{\partial I_\delta\}_{\delta \in \A} = \emptyset$. 

By \eqref{eq:inverse_iterates} and Lemma \ref{lem: local_reflection}, 
$T^{-k}(c_\sigma) \in J$ if and only if $T^{k - \delta_{1/2}}(c_\sigma) \in J$, 
for any $0 \leq k \leq \ell$. Hence, the first visit time of $c_\sigma$ to $J$ 
via $T$ is $\ell - \delta_{1/2}(\sigma)$, which implies $$
 T_J(p_J(c_\sigma)) = T^{\ell - \delta_{1/2}(\sigma)}(c_\sigma).$$
Moreover, $p_J(c_\sigma) = T^{-\ell}(c_\sigma)$ is a fixed point of $\I_J \circ T_J$. Indeed, by \eqref{eq:generalized_conjugacy} and Lemma \ref{lem: local_reflection},
\[\I_J \circ T_J(p_J(c_\sigma)) = \I \circ T \big( T^{\ell - 
\delta_{1/2}(\sigma)}(c_\sigma) \big) = T^{-\ell + \delta_{1/2}(\sigma)}\circ 
T^{-1} \circ \I (c_\sigma) = T^{-\ell + \delta_{1/2}(\sigma)} (c_\sigma) = 
p_J(c_\sigma). \]
Therefore, since $p_J(c_\sigma)$ is a fixed point of $\I_J \circ T_J$ and $T_J$ exchanges its continuity intervals symmetrically, $p_J(c_\sigma) = \hat c_\gamma^J,$ for some $\gamma \in \mathcal C.$

Now, let us assume that $\{c_\sigma, T(c_\sigma), \ldots, T^{\ell - \delta_{1/2}(\sigma)}(c_\sigma)\} \cap \{\partial I_\delta\}_{\delta \in \A} \neq \emptyset$. 

Let $0 \leq k \leq \ell - \delta_{1/2}(\sigma)$ be the minimum such that $T^k(c_\sigma) = \partial I _\delta$ for some $\delta \in \A$. By Lemma \ref{lem:symmetric_connection}, $T^{-k - \delta_{1/2}(\sigma)}(c_\sigma) = \partial I_{\hat \delta},$ where $\pi_0(\hat \delta) = \pi_0(\delta) - 1$. Since $-k - \delta_{1/2}(\sigma) \geq -\ell$, it follows that $p_J(\partial I_{\hat \delta}) = p_J(c_\sigma)$ which, by definition of $T_J$ (see Section \ref{sc:induced_IETs} and Lemma \ref{lem: alphabetafetrconnections}), coincides with $\partial I_\gamma^J$, for some $\gamma \in \A_J$. Notice that, in this case, $c_\sigma$ lies in a non-trivial connection of $T$.
\end{proof}

The claim above shows that $p_J$ maps the set $\{c _\sigma \mid \sigma \in \mathcal B\},$ where $$\mathcal{B} := \{\sigma \in \A \cup \{\tfrac{1}{2}\} \mid \sigma \neq \beta \text{ and } c_\sigma \text{ does not belong to any non-trivial connection}\},$$
injectively to the set $\{\hat c_\gamma^J \mid \gamma \in \mathcal C\}.$ 

Indeed, if for $\sigma, \sigma' \in \mathcal B$ we have $p_J(c_\sigma) = \hat c_\gamma^J = p_J(c_\sigma')$ then it follows from the previous claim that $r_J(\hat c_\gamma^J) = 2b_J(c_\sigma) - \delta_{1/2}(\sigma) = 2b_J(c_\sigma') - \delta_{1/2}(\sigma')$. Hence, since all the terms in the previous equality must have the same parity, it follows that either $\sigma = \tfrac{1}{2} = \sigma'$ or $\sigma, \sigma' \in \A$. In the latter case, it follows from the claim that $c_\sigma = T^{r_J( \hat c_\gamma^J)/2}(\hat c_\gamma^J) = c_\sigma'.$

Therefore, Claim \ref{cl:backward_centers} implies that $\# \mathcal C \geq \# \mathcal B$. Since, by Corollary \ref{cor:centers_connections}, any non-trivial connection contains at most one point of the form $\{c_\sigma \mid \sigma \in \A \cup \{\tfrac{1}{2}\} \}$ it follows that $\#\mathcal B \geq d - d'$. Therefore $\#\mathcal C \geq d - d'$, which together with $\#\mathcal C\le\#\mathcal \A_J = d - d' \le\#\mathcal A$, implies $$\# \mathcal C = \#\mathcal B = d - d' = \# \A_J.$$
\end{proof}

\begin{proof}[Proof of Proposition \ref{prop:induced_centers}]
By Proposition \ref{prop:induced_symmetric}, it follows that the maximal continuity intervals of $T_J$, which in the proof of Proposition \ref{prop:induced_symmetric} we denoted by $\{\hat I^J_{\alpha}\}_{\alpha\in\mathcal C}$, coincide with the intervals exchanged by $T_J$, which we denoted by $\{I^J_{\alpha}\}_{\alpha\in\mathcal \A_J}$. 

Therefore, Proposition \ref{prop:induced_centers} follows from Claim \ref{cl:backward_centers} and the fact that $p_J$ maps the set $$\{c _\sigma \mid \sigma \in \A \cup \{\tfrac{1}{2}\},\, \sigma \neq \beta \text{ and } c_\sigma \text{ does not belong to any non-trivial connection}\},$$
injectively to the set $\{\hat c_\gamma^J \mid \gamma \in \mathcal C\} = \{c_\gamma^J \mid \gamma \in \mathcal \A_J\},$ which we showed and the end of the proof of Proposition \ref{prop:induced_symmetric}.
\end{proof}

The following is a direct consequence of Proposition \ref{prop:induced_centers}.

\begin{corollary}\label{cor: emptyconnection}
If a symmetric IET $T$ has a connection that does not contain a point from the set $\{c_{\alpha} \mid \alpha \in \A \cup \tfrac{1}{2}\}$, then $T$ is not ergodic.
\end{corollary}

The following example illustrates the situation described in the previous corollary.

\begin{example}\label{ex:1}
 Let $\mathcal{A}=\{1,2,3,4\}$, and let $T$ be a symmetric $4$-IET with permutation $\pi_0(i)=i,$\, $1\le i\le 4$ and lengths $|I_i|=\lambda_i>0,$ for $i=1,2,3$, and $|I_4|=2(\lambda_1+\lambda_2)+\lambda_3$. Choose $\lambda = (\lambda_i)_{i \in \A}$ such that $|\lambda|_1 = 1$. Note that $T(\partial I_2)= \lambda_3 + \lambda_4 = 1- (\lambda_1+\lambda_2)$ and is not the middle point of $I_4$. Also note that $T^2(\partial I_2)=\lambda_1+\lambda_2+\lambda _3=\partial I_4<\frac{1}{2}$. Hence $T$ has a connection (of length 2), which does not contain any center point or $\frac{1}{2}$. In this case, we have an invariant set $I_3\cup T(I_3)$, where $T^2(I_3)=I_3$. Thus $T$ is not ergodic. This is represented in the figure below.
\end{example}
\begin{figure}[h]
\centering
\begin{tikzpicture}
 \draw[thick] (0,0) -- (10,0);

 \draw[thick, blue] (0,0) -- (1,0) node[midway, above] {$I_1$};
 \draw[thick, red] (1,0) -- (2.5,0) node[midway, above] {$I_2$};
 \draw[thick, green] (2.5,0) -- (3.75,0) node[midway, above] {$I_3$};
 \draw[thick, orange] (3.75,0) -- (10,0) node[midway, above] {$I_4$};

 \draw (1,0.1) -- (1,-0.1) node[below] {{\pgfmathparse{0.1}\pgfmathresult}};
 \draw (10,0.1) -- (10,-0.1) node[below] {{\pgfmathparse{1}\pgfmathresult}};
 \draw (0,0.1) -- (0,-0.1) node[below] {{\pgfmathparse{0}\pgfmathresult}};
 \foreach \x in {2.5,3.75}
 \draw (\x,0.1) -- (\x,-0.1) node[below] {};

 \draw[thick] (0,-1.5) -- (10,-1.5);

 \draw[thick, orange] (0,-1.5) -- (6.25,-1.5) node[midway, above] {$T(I_4)$};
 \draw[thick, green] (6.25,-1.5) -- (7.5,-1.5) node[midway, above] {$T(I_3)$};
 \draw[thick, red] (7.5,-1.5) -- (9,-1.5) node[midway, above] {$T(I_2)$};
 \draw[thick, blue] (9,-1.5) -- (10,-1.5) node[midway, above] {$T(I_1)$};
 \draw (10,0.1-1.5) -- (10,-0.1-1.5) node[below] {{\pgfmathparse{1}\pgfmathresult}};
 \draw (0,0.1-1.5) -- (0,-0.1-1.5) node[below] {};
 \foreach \x in {6.25,7.5,9}
 \draw (\x,0.1-1.5) -- (\x,-0.1-1.5) node[below] {};

 \draw[thick] (0,-3) -- (10,-3);

 \draw[thick, green] (2.5,-3) -- (3.75,-3) node[midway, above] {$T^2(I_3)$};
 \draw[thick, red] (3.75,-3) -- (5.25,-3) node[midway, above] {$T^2(I_2)$};
 \draw (10,-2.9) -- (10,-3.1) node[below] {{\pgfmathparse{1}\pgfmathresult}};
 \draw (0,-2.9) -- (0,-3.1) node[below] {{\pgfmathparse{0}\pgfmathresult}};
 
 \foreach \x in {2.5,3.75,5.25}
 \draw (\x,-2.9) -- (\x,-3.1) node[below] {};
 \draw[thick, dotted] (2.5,-3) -- (2.5,-0.1);
 \draw[thick, dotted] (3.75,-3) -- (3.75,-0.1);

\end{tikzpicture}
\caption{Plot of the exchanged intervals (and some of their iterates) for the symmetric IET $T$ described in Example \ref{ex:1}. The set $\{\partial I_2, T(\partial I_2), T^2(\partial I_2) = \partial I_4\}$ defines a connection disjoint from the set $\{c_{\alpha} \mid \alpha \in \A \cup \tfrac{1}{2}\}$. By Corollary \ref{cor: emptyconnection}, $T$ is not ergodic.}
\end{figure}
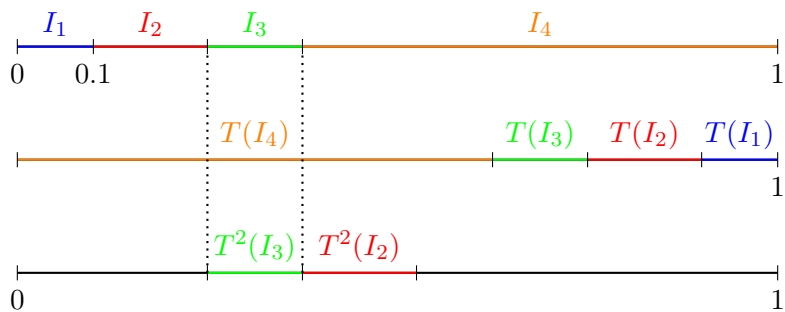

The following corollary is of independent interest.
\begin{corollary}\label{cor: evalue-1}
 Assume that $T$ is an ergodic symmetric IET such that $c_{1/2}$ lies inside a connection. Then $-1$ is an eigenvalue for the Koopman operator associated with $T$. In particular, $T$ is not weak mixing. 
\end{corollary}
\begin{proof}
Let $J \subseteq I$ as in Proposition \ref{prop:induced_symmetric} such that $J$ does not contain any point from any connection. Such an interval exists by Lemma \ref{lem:sym_int_disjoint}. Then, by Propositions \ref{prop:induced_symmetric} and \ref{prop:induced_centers}, $T_J$ is symmetric and the middle points $\{c_\gamma^J\}_{\gamma \in \A_J}$ of the exchanged intervals $\{I_\gamma^J\}_{\gamma \in \A_J}$ are preimages of the middle points of the intervals exchanged by $T_J$. 
Moreover, again by Proposition \ref{prop:induced_centers}, the Rokhlin towers associated to $T_J$ are all of even height since for any $\gamma \in \A_J$ there exists $\sigma \in \A$ such that $c_\gamma^J = p_J(c_\sigma)$ and $T_J(c_\gamma^J) = T^{2b_J(c_\sigma)}(c_\gamma^J)$, where $b_J$ is given by \eqref{eq:backward_return}.

By defining a function $f$ that equals $1$ (resp. $-1$) on the odd (resp. even) levels of each Rokhlin tower, we get an eigenfunction of $T$ with eigenvalue $-1$, that is, such that $f \circ T = -f$.
\end{proof}

\begin{example}\label{ex:2}
 Let $\mathcal{A}=\{1,2,3,4\}$, and let $T$ be a symmetric $4$-IET with permutation $\pi_0(i)=i,$\, $1\le i\le 4$ and lengths $|I_i|=\lambda_i>0,$ for $i=1,2,3$, and $|I_4|=\frac{1}{2}+\lambda_1+\lambda_2$, where $2(\lambda_1+\lambda_2)+\lambda_3=\frac{1}{2}$ and $\lambda_2>\lambda_1+\lambda_3$. With this configuration, we have that $T(\partial I_2)=1-\lambda_1-\lambda_2$, $T^2(\partial I_2)=\frac{1}{2}$, and $T^3(\partial I_2)=\partial I_3$. Thus we have a connection containing $\frac{1}{2}$. 
 
 By inducing on the interval $J=I_2=[\partial I_2, \partial I_3)$, 
 we obtain a symmetric 3-IET, with initial order of $J_1, J_2, J_3$ 
 and the parameters as follows: $|J_1|=\lambda_3, |J_2|=\lambda_1, 
 |J_3|=\lambda_2-\lambda_1-\lambda_3>0$. 
 
As shown in the figure below, we can check that $T^4(J_3)\subseteq I_2$, 
 $T^{3}(J_1)=I_3$, $T^3(I_3)\subseteq I_2$, $T^{4}(J_2)=I_1$, $T^4(I_1)\subseteq 
 I_2$. Thus the first return times are $r_J(J_1)=6, r_J(J_2)=8, r_J(J_3)=4$. 
 So all the towers are of even heights, and we can give value $1$ to the 
 odd levels and $-1$ to the even levels of each tower, which gives us an 
 eigenfunction of eigenvalue -1. To guarantee the ergodicity of the IETs, it 
 suffices to ask that $\lambda_2$ and $\lambda_3$ are rationally 
 independent. {Indeed, then the induced map is a 3-IET, whose image after a 
 single step of the classical Rauzy-Veech induction yields an irrational 
 rotation.}

\end{example}

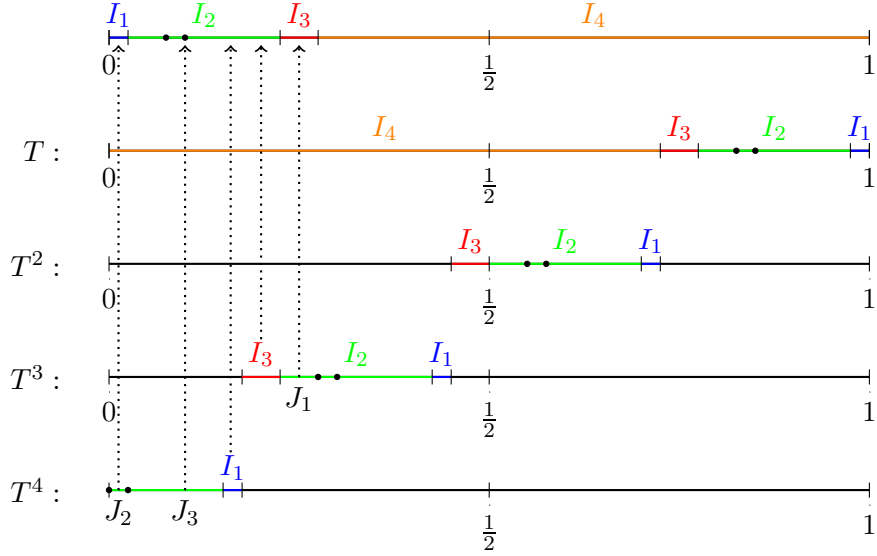
\begin{figure}[h]
 \centering
 \begin{tikzpicture}
 \draw[thick] (0,0) -- (10,0);

 \draw[thick, blue] (0,0) -- (0.25,0) node[midway, above] {$I_1$};
 \draw[thick, green] (0.25,0) -- (2.25,0) node[midway, above] {$I_2$};
 \draw[thick, red] (2.25,0) -- (2.75,0) node[midway, above] {$I_3$};
 \draw[thick, orange] (2.75,0) -- (10,0) node[midway, above] {$I_4$};

 \filldraw (0.75,0) circle (1pt);
 \filldraw (2.0-1,0) circle (1pt);
 
 \foreach \x in {0, 0.25, 2.25, 2.75, 10}
 \draw (\x,0.1) -- (\x,-0.1) node[below] {};
 \draw (5,0.1) -- (5,-0.1) node[below] {$\frac{1}{2}$};
 \draw (10,0.1) -- (10,-0.1) node[below] {1};
 \draw (0,0.1) -- (0,-0.1) node[below] {0};

 \node[left] at (-0.5, -1.5) {\( T: \)};
 \draw[thick] (0,-1.5) -- (10,-1.5);
 \draw[thick, orange] (0, -1.5) -- (7.25, -1.5) node[midway, above] {$I_4$};
 \draw[thick, red] (7.25, -1.5) -- (7.75, -1.5) node[midway, above] {$I_3$};
 \draw[thick, green] (7.75, -1.5) -- (9.75, -1.5) node[midway, above] {$I_2$};
 \draw[thick, blue] (9.75, -1.5) -- (10, -1.5) node[midway, above] {$I_1$};
 \foreach \x in {0, 7.25, 7.75, 9.75, 10}
 \draw (\x, -1.4) -- (\x, -1.6) node[below] {};
 \draw (5, -1.4) -- (5, -1.6) node[below] {$\frac{1}{2}$};
 \draw (10, -1.4) -- (10, -1.6) node[below] {1};
 \draw (0, -1.4) -- (0, -1.6) node[below] {0};

 \filldraw (0.75+7.5,-1.5)circle (1pt);
 \filldraw (2.0+7.5-1,-1.5) circle (1pt);

 \node[left] at (-0.5, -3) {\( T^2: \)};
 \draw[thick] (0,-3) -- (10,-3);
 \draw[thick, green] (5, -3) -- (7, -3) node[midway, above] {$ I_2$};
 \draw[thick, blue] (7, -3) -- (7.25, -3) node[midway, above] {$I_1$};
 \draw[thick, red] (4.5, -3) -- (5, -3) node[midway, above] {$I_3$};

 \foreach \x in {0, 4.5, 5, 7, 7.25, 10}
 \draw (\x, -2.9) -- (\x, -3.1) node[below] {};
 \draw (5, -3.2) -- (5, -3.2) node[below] {$\frac{1}{2}$};
 \draw (10, -3.2) -- (10, -3.2) node[below] {1};
 \draw (0, -3.2) -- (0, -3.2) node[below] {0};
 
 \filldraw (0.75+4.75,-3) circle (1pt);
 \filldraw (2.0+4.75-1,-3) circle (1pt);
 
 \node[left] at (-0.5, -4.5) {\( T^3:\)};
 \draw[thick] (0,-4.5) -- (10,-4.5);
 \draw[thick, green] (2.25, -4.5) -- (4.25, -4.5) node[midway, above] {$I_2$};
 \draw[thick, blue] (4.25, -4.5) -- (4.5, -4.5) node[midway, above] {$I_1$} ;
 \filldraw (0.75+2,-4.5) circle (1pt);
 \filldraw (2.0+2-1,-4.5) circle (1pt);

 \foreach \x in {0, 1.75, 2.25, 4.25,4.5, 5, 10}
 \draw (\x, -4.6) -- (\x, -4.4) node[below] {};
 \draw[thick, red] (1.75, -4.5) -- (2.25, -4.5) node[midway, above] {$I_3$};
 \draw (5, -4.7) -- (5, -4.7) node[below] {$\frac{1}{2}$};
 \draw (10, -4.7) -- (10, -4.7) node[below] {1};
 \draw (0, -4.7) -- (0, -4.7) node[below] {0};
 \draw[](2.5, -4.5) -- (2.5, -4.5) node[below] {$J_1$};

 \node[left] at (-0.5, -6) {\( T^4: \)};
 \draw[thick] (0,-6) -- (10,-6);
 \draw[thick, green] (0,-6) -- (1.5,-6);
 \draw[](0.125, -6) -- (0.125, -6) node[below] {$J_2$};
 \draw[thick, blue] (1.5,-6) -- (1.75,-6) node[midway,above] {$I_1$} ;
 \foreach \x in {0, 1.5, 5, 1.75, 10}
 \draw (\x, -5.9) -- (\x, -6.1) node[below] {};
 \draw (5, -6.2) -- (5, -6.2) node[below] {$\frac{1}{2}$};
 \draw (10, -6.2) -- (10, -6.2) node[below] {1};
 \filldraw (0,-6) circle (1pt);
 \filldraw (1.25-1,-6) circle (1pt);
 \draw[](1, -6) -- (1, -6) node[below] {$J_3$};

 \draw[thick, dotted, ->] (2.5,-4.5) -- (2.5,-0.1);
 \draw[thick, dotted, ->] (2,-4) -- (2,-0.1);

 \draw[thick, dotted, ->] (0.125,-6) -- (0.125,-0.1);
 \draw[thick, dotted, ->] (1,-6) -- (1,-0.1);
 \draw[thick, dotted, ->] (1.6,-5.5) -- (1.6,-0.1);

 \end{tikzpicture}
\caption{Plot of the exchanged intervals (and some of their iterates) for the symmetric IET $T$ described in Example \ref{ex:2}. Every Rohlin tower in the decomposition associated with the induced map $T_J$ (see \eqref{eq:towers_decomposition}) is of even height.}
 \label{}
\end{figure}

\section{The essential values criterion}
In this section, we recall and state the standard notion of essential value, which is a classical tool to study skew products' ergodicity. Let $(X,\mathcal B, \mu)$ be a standard probability space. Let $T:X\to X$ be $\mu$-measure preserving automorphism and let $f:X\to \R^m$, where $m \geq 1$. Consider the skew product $T_f$ on $X\times \R^n$ given by
\[
T_f(x,r)=(Tx,r+f(x)).
\]
We say that $a\in\R^m$ is an \emph{essential value} of $T_f$ if for every $\epsilon>0$ and every measurable subset $E\subseteq X$ with $\mu(E)>0$, there exists $n\in\N$ such that 
\[
\mu\{x\in E\mid T^{n}(x)\in E\text{ and }|S_n f(x)-a|<\epsilon\}>0,
\]
where $S_n f (x)$ denotes the $n$-th Birkhoff sum of $f$ evaluated at $x$, which is given by
\begin{equation}
	\label{eq:Birkhoff_sums}
	S_n f (x):=\begin{cases}
		\sum_{i=0}^{n-1}f(T^{i}(x))& \text{ if }n\ge 1,\\
		0&\text{ if }n=0,\\
		-\sum_{i=-n}^{-1}f(T^{i}(x))&\text{ if }n\le 1.
	\end{cases}
\end{equation}
We denote the set of essential values of $T_f$ by $Ess(T_f)$.

The following classical fact links this notion to the ergodicity of the skew product.
\begin{theorem}\label{thm: essvalues}
 With the notation as above, the set $Ess(T_f)$ is a closed subgroup of $\R^m$. Moreover, $T_f$ is ergodic w.r.t. $\mu\otimes Leb_{\R^m}$ if and only if $T$ is ergodic and $Ess(T_f)=\R^m$.
\end{theorem}
We will use a simplified version of this criterion, as introduced by Conze and Fr\k{a}czek in 
\cite{conze_cocycles_2011}.
\begin{theorem}[Lemma 2.7 in \cite{conze_cocycles_2011}]\label{thm: FUcriterion}
 Let $a\in\R^m$. Assume that for every $\epsilon>0$ there exists a sequence 
 of subsets 
 $\{\Xi_n\}_{n\in\N}$ and an increasing sequence 
 $\{q_n\}_{n\in\N}$ of natural numbers such that 
 \begin{enumerate}
 \item \label{cond:measure} $\liminf_{n\to\infty}\mu(\Xi_n)>0$,
 \item \label{cond:rigidity}$\lim_{n\to\infty}\sup_{x\in \Xi_n}|T^{q_n}(x)-x|=0$,
 \item \label{cond:almost_invariant} $\lim_{n\to\infty}\mu(\Xi_n\triangle T(\Xi_n))=0$,
 \item \label{cond:BS} $|S_{q_n}f(x)-a|<\epsilon$.
 \end{enumerate}
 Then $a\in Ess(T_f)$. 
\end{theorem}
{The Lemma 2.7 in \cite{conze_cocycles_2011} actually states that 
the 
topological support of the limit distribution 
$P:=\lim_{n\to\infty}\frac{1}{\mu(\Xi_n)}(S_{q_n}f(x)|_{\Xi_n})_*\mu|_{\Xi_n}$, 
which exists up 
to taking a subsequence due to tightness guaranteed by Condition 
\eqref{cond:BS}, is contained in $Ess(T_f)$. However, again by \eqref{cond:BS}, 
the topological support of $P$ is contained in $[a-\epsilon,a+\epsilon]$. By 
passing with $\epsilon$ to 0 and by the fact that $Ess(T_f)$ is a closed subset 
of $\R^m$, we get that $a\in Ess(T_f)$.}

We now provide a version of the above criterion, that is going to be effective for our purposes.

\begin{proposition}\label{prop: effcriterion}
 Let $T:I\to I$ be an ergodic IET and let $m\in \N_+$. Assume that the function $f:I^{\times m}\to \R^{m}$ satisfies the following: 
 \begin{enumerate}[(i)]
 \item \label{cond:continuity} $f$ is of the form $(x_1,\ldots,x_m)\mapsto (f_1(x_1),\ldots,f_m(x_m))$ with each $f_j$ being continuous over exchanged intervals,
 \item \label{cond:sequences} there exist sequences $\{\Xi_n\}_{n\in\N}$ and $\{q_n\}_{n\in\N}$ satisfying \eqref{cond:measure}-\eqref{cond:almost_invariant} in Theorem \ref{thm: FUcriterion}, with the additional assumption that, for any $n \in \N$, $\Xi_n=\bigsqcup_{i=0}^{h_n-1}T^{i}({I_n})$ is a Rokhlin towers of $h_n\le q_n$ intervals with $|I_n|\le\frac{1}{q_n}$ and $\sup_{x\in \Xi_n}|T^{q_n}(x)-x|\le D/q_n$ for some $D>1$,
 \item \label{cond:cont_interval} for every $x\in\Xi_n$, the interval $[x,T^{q_n}(x)]$ (or $[T^{q_n}(x),x]$) is a continuity interval of $f_j$, for every $j\in\{1,\ldots,m\}$,
 \item \label{cond:derivative} there exists $C>1$ such that
 \[
 |S_{q_n}f_j(x)|\le C, \qquad C^{-1}q_n\le S_{q_n}f_j'(x)\le Cq_n,
 \]
 for any $j=1,\ldots,m$ and any $x\in\Xi_n$, and 
 \[
 |f_j'(x)|\le C,
 \]
 for any $x\in I$.
 \end{enumerate}
 If additionally $T^{\otimes m}$ is ergodic, then so is $T_f^{\otimes m}$.
\end{proposition}
\begin{proof}
 In view of Theorem \ref{thm: essvalues}, it is enough to show that we can 
 obtain an $m$-dimensional cube of essential values. We will prove the 
 proposition for $m=1$ by realizing the essential values through Roklhin 
 towers, which are subsets of $\Xi_n$. For $m\ge 2$ we first obtain the 
 same result for every $j=1,\ldots,m$, i.e. we find an interval $(a_j,b_j)$ 
 of essential values, realized through Rokhlin towers inside $\Xi_n$. Hence, 
 the set $(a_1,b_1)\times\ldots\times(a_m,b_m)$ is the set of essential 
 values realized via sequences of Rokhlin towers inside 
 $\bigsqcup_{i=0}^{h_n-1}(T^{\otimes m})^i(I_n^{\times m})$. 

 That being said, we assume from now on that $m=1$. Note that due to our assumption on $\Xi_n$ and the derivative of $f$, each of the sets $S_{q_n}f(\Xi_n)\subseteq [-2C,2C]$ is a uniformly bounded union of intervals. Note that, since 
\[\liminf_{n\to\infty} h_n |I_n| = \liminf_{n\to\infty}\mu(\Xi_n) > 0 \] and $h_n\le q_n$, we have that, by passing to a subsequence if necessary, there exists $E>1$ with
\[
E^{-1}\le q_n|I_n|\le 1.
\]
Hence, given the assumption on the derivative, we have 
\[
\textup{Leb}_{\R}\left(S_{q_n}f(\Xi_n)\right)\ge \frac{1}{CE}>0,
\]
for every $n\in\N$. 
By passing to a subsequence if necessary, we may assume that the sets $S_{q_n}f(\Xi_n)$ have a nonempty intersection. Let $y\in \bigcap_{i=1}^{\infty} S_{q_n}f(\Xi_n)$. We claim that $y$ is an essential value. For this purpose, we will construct a subtower $\Xi_{n}^{y}\subseteq \Xi_n$, depending on $\epsilon>0$, such that \eqref{cond:measure} and \eqref{cond:BS} in Theorem \ref{thm: FUcriterion} is satisfied. This is enough, since \eqref{cond:rigidity} and \eqref{cond:almost_invariant} are automatically satisfied by any sequence of subtowers of $\Xi_n$. 

Fix $\epsilon>0$ and consider the point {$x_n \in \Xi_n$ such that $S_{q_n}f(x_n)=y$}. 
Let $\ell_n\in\{0,\ldots,h_n-1\}$ be such that $x_n\in T^{\ell_n}(I_n)$ and assume WLOG that $\ell_n<h_{n}/2$, the other case being treated symmetrically. Since each level of the tower is an interval then either 
\[\text{$\left(x_n,x_n+\frac{1}{\max\{C, D,E\} q_n}\right)\subseteq T^{\ell_n}(I_n)$ $\quad$ or $\quad$ $\left(x_n-\frac{1}{\max\{C,D,E\} q_n},x_n\right)\subseteq T^{\ell_n}(I_n)$}.
\]
Again, WLOG, let us assume that it is the former. Consider the tower 
\[
\Xi_n^y:=\bigsqcup_{i=0}^{\epsilon h_n/4CD} T^i\left(x_n,x_n+\frac{\epsilon}{2\max\{C,D,E\} q_n}\right)\subseteq \Xi_n.
\]
We now show that for every $x\in \Xi_n^y$ we have $S_{q_n}f(x)\in(y-\epsilon,y+\epsilon)$. If $x\in T^{\ell_n}(I_n)$, then by the mean value theorem, we have
\[
\left|S_{q_n}f(x)-y\right|=\left|S_{q_n}f(x)-S_{q_n}f(x_n)\right|\le Cq_n|x-x_n|\le \epsilon/2.
\]
If $x\in T^{j} (T^{\ell_n}(I_n))$ with $j=1,\ldots,\epsilon h_n/4CD$, then we 
split the Birkhoff sum into two pieces $S_{q_n}f(x) = S_{q_n - j}f (x) + 
S_jf(T^{q_n - j}(x))$, which we estimate separately. For the first term, we have
\[
\left|S_{q_n-j} f(x)-S_{q_n-j} f(T^j(x_n))\right|\le Cq_n|x-T^j(x_n)|\le\epsilon/2.
\]

For every $k\in\{0,\ldots,h_n/2\}$ we have that $T^{k}(x_n)$ and 
$T^{k}(T^{-j}(x))$ belong to the same level of $\Xi_n$ and, since $\Xi_n$ is a 
tower of intervals as such, belong to the same interval continuity of $f$. 
Moreover, by \eqref{cond:cont_interval},
 $T^k(T^{q_n-j}(x))=T^k(T^{q_n}(T^{-j}(x)))$ and 
$T^{k}(T^{-j}(x))$ belong to the same 
continuity interval of $f$ for every $k\in\{0,\ldots,h_n/2\}$. Hence 
$T^k(T^{q_n-j}(x))$ and $T^k(x_n)$ belong to the same continuity intervals of 
$f$. 
Hence, by \eqref{cond:BS} and mean value theorem we 
get
\[
\left|S_j f(T^{q_n-j}(x))-S_j f(x_n)\right|\le \frac{C\epsilon h_n}{4D}|T^{q_n-j}(x)-x_n|\le \frac{C\epsilon h_n}{4CD} \frac{2D}{q_n}\le\epsilon/2
\]
and thus $|S_{q_n} f(x)-y|\le \epsilon$. It remains to notice that 
\[
\liminf_{n\to\infty}\textup{Leb}(\Xi_n^y)\ge \liminf_{n\to\infty}\frac{\epsilon h_n}{4CD}\frac{1}{\max\{C,D,E\}}|I_n|
\ge \liminf_{n\to\infty} \frac{1}{4(\max\{C,D,E\})^2}\textup{Leb}(\Xi_n)>0.
\]
Thus we have proved that $y\in Ess(T_f)$. Note that for every $n\in\N$ and every 
\[
x\in \left(x_n,x_n+\frac{1}{2\max\{C,D,E\} q_n}\right).
\]
Hence, 
\[
S_{q_n}f(x)-S_{q_n}f(x_n)\ge C^{-1}q_n |x-x_n|.
\]
Therefore, by applying similar reasoning as to $y$, we obtain that every $z\in\left[y,y+\frac{1}{2C^{-1}\max\{C,D,E\}}\right]$ is an essential value of $T_f$. This finishes the proof of the proposition.
\end{proof}
\section{Proofs of main results.}

This section contains the proof of Theorems \ref{thm: ergo}, \ref{thm: uniqergo} and \ref{thm: index}. In all three proofs, we will apply the ergodicity criterion described in Proposition \ref{prop: effcriterion}. For this reason, we start this section by outlining a construction that will be common to all of the proofs since it concerns only the underlying IET $T$ and not the cocycle being considered. At the end of this construction, we will describe in detail why the assumptions of Proposition \ref{prop: effcriterion} are fulfilled in each setting.

Throughout this section, let $T=(\pi,\lambda): I \to I$ be an ergodic symmetric IET on $d = \#\A$ intervals $\{I_\alpha\}_{\alpha \in \A}$.

By Corollary \ref{cor:centers_connections}, there exists $\beta \in \A$ such that $c_{\beta}$ is not a part of any connection of $T$. By Lemma \ref{lem:sym_int_disjoint}, there exists $\alpha \in \A$ and a nested sequence of $\beta$-symmetric intervals $\{J_n\}_{n \in \N}$ disjoint from the connections of $T$, with endpoints $T^{-m_n}(\partial I_{\alpha})$ and $T^{m_n}(\partial I_{\hat\alpha})$ for some $m_n \nearrow \infty$, where $\pi_0(\hat\alpha) = \pi_0(\alpha)-1$, and such that $\{c_\beta\} = \bigcap_{n \in \N} J_n$.

By Proposition \ref{prop:induced_symmetric}, for every $n\in\N$, the induced IET $T_{J_n}$ is a symmetric IET with $d - d'$ intervals, where $d'$ is the number of non-trivial connections of $T$. WLOG we may index their exchanged intervals using the same alphabet $\B \subseteq \A$. Let us denote by $\{I_\gamma^n\}_{\gamma \in \B}$ the intervals exchanged by $T_{J_n}$ and by $\{c_\gamma^n\}_{\gamma \in \B}$ their middle points, where, to avoid the use of double subscripts, we changed our usual notation $I_\gamma^{J_n}$ (resp. $c_\gamma^{J_n}$) to $I_\gamma^n$ (resp. $c_\gamma^n$).

By Proposition \ref{prop:induced_centers}, each of the towers in the decomposition of $I$ associated with $T_{J_n}$ 
\begin{equation}
\label{eq:decomposition_sequence}
I = \bigsqcup_{\gamma \in \B} \bigsqcup_{i = 0}^{h_\gamma^n - 1} T^i(I_\gamma^n),
\end{equation}
where $h^n = (h^n_\gamma)_{\gamma \in \B}$ is some vector in $\N^\A_+$ (see Section \ref{sc:induced_IETs}), contains exactly one point from the set
\[ \{c _\sigma \mid \sigma \in \A \cup \{\tfrac{1}{2}\},\, \sigma \neq \beta \text{ and } c_\sigma \text{ does not belong to any non-trivial connection}\}, \]
in the middle of its \emph{central level} $T^{\lfloor h_\gamma^n/2 
\rfloor}(I^n_\gamma)$. More precisely, for every $\gamma \in \B$ there exists 
$\sigma$ in the set above such that the middle point $c_\gamma^n$ of the 
interval $I_\gamma^n$ verifies $c_\sigma = T^{\lfloor h_\gamma^n/2 
\rfloor}(c_\gamma^n).$

Using these facts, we will show how to build towers $\{\Xi_n\}_{n\in\N}$ and a sequence $\{q_n\}_{n\in\N}$ that satisfies the assumptions of Proposition \ref{prop: effcriterion}.

A natural approach would be to consider subtowers of the already constructed towers. However in their current form, the towers may be very unbalanced: the wide towers may be very short and thus of very small measure, while thin towers may be very tall and contain the majority of the interval $I$. Since we need to construct towers of measure bounded away from 0, we would have to choose them to be inside of the thin towers. This however makes it very difficult to control the rigidity of $T$ inside those towers as well as to estimate the values of Birkhoff sums. We tackle this problem by jumping between the points around which we induce.

Consider a Rokhlin tower $X_n:=\bigsqcup_{i=0}^{h_{\gamma_n}^n-1}T^i(I^n_{\gamma_n})$, where $\gamma_n \in\mathcal \B$ is chosen so that its Lebesgue measure is the largest compared to the other towers in the decomposition \eqref{eq:decomposition_sequence}. In particular
\[
\textup{Leb}(X_n)\geq \frac{1}{\#\B}\ge\frac{1}{\#\A}.
\]

Let us denote by $\mathfrak J_n$ the central level of this tower and recall that it contains a point $c_\sigma$ for some $\sigma \in \A \cup \{\tfrac{1}{2}\}$. By Proposition \ref{prop:induced_symmetric}, the induced transformation $T_{\mathfrak J_n}$ is a symmetric IET with $d - d'$ intervals, and we denote its exchanged intervals by $\{\mathfrak I_{\gamma}^n\}_{\gamma \in \mathcal B}$. As before, we have a decomposition in Rokhlin towers of the form
\[ I = \bigsqcup_{\gamma \in \B} \bigsqcup_{i = 0}^{\mathfrak h_\gamma^n - 1} T^i(\mathfrak I_\gamma^n).\]

Let $\Gamma_n \in \B$ be such that $\mathfrak I_{\Gamma_n}^n$ is the largest of all intervals exchanged by $T_{\mathfrak J_n}$. Up to taking a subsequence, let us assume WLOG that there exists $\Gamma \in \B$ such that $\Gamma_n = \Gamma$, for every $n \in \N$.

As before, by Proposition \ref{prop:induced_centers}, the tower $\bigsqcup_{i=0}^{\mathfrak h_{\Gamma}^n-1}T^i(\mathfrak I_{\Gamma}^n)$ contains exactly one point of the form $c_{\sigma}$ for some $\sigma \in \A \cup \{\tfrac{1}{2}\}$ in the middle of its central level. Let us denote this point by $\mathfrak c_n$. 

Define 
\[
\Xi_n:=\bigsqcup_{i=0}^{h_{\gamma_n}^n/2-1}T^{i}(\mathfrak I_{\Gamma}^n)\qquad \text{and}\qquad q_n:=\mathfrak h_{\Gamma}^n.
\]

Before passing to the proofs of each of the theorems, let us show that $\{\Xi_n\}_{n\in\N}$ and $\{q_n\}_{n\in\N}$ satisfy the assumptions \eqref{cond:sequences} and \eqref{cond:cont_interval} in Proposition \ref{prop: effcriterion}. 

We start by showing that \eqref{cond:sequences} in Proposition \ref{prop: effcriterion} is satisfied, that is, that the sequences above verify \eqref{cond:measure}-\eqref{cond:almost_invariant} in Theorem \ref{thm: FUcriterion}.

First, we can easily check that $\{\Xi_n\}_{n \in \N}$ satisfies \eqref{cond:measure} in Theorem \ref{thm: FUcriterion}. Indeed, 
\begin{align*}
\textup{Leb}(\Xi_n) & =(h_{\gamma_n}^n/2)|\mathfrak I_{\Gamma}^n| \\ & > \frac{1}{\#\A}(h_{\gamma_n}^n/2)|\mathfrak J_n|=\frac{1}{\#\A}(h_{\gamma_n}^n/2)|I^n_{\gamma}| \\
& > \frac{1}{3\#\A} h_{\gamma_n}^n|I^n_{\gamma}|=\frac{1}{3\#\A} \textup{Leb}(X_n) \\
& >\frac{1}{3\#\A^2}.
\end{align*}

To see that $\{\Xi_n\}_{n\in\N}$ satisfy \eqref{cond:almost_invariant} in Theorem \ref{thm: FUcriterion} it is enough to notice that 
\[
\mu(\Xi_n\triangle T(\Xi_n))\le 2|\mathfrak I_{\Gamma}^n|\to 0 \text{ as }n\to\infty. 
\]

We will now show that 
\begin{equation}\label{eq: derbound}
\sup_{x\in\Xi_n}|T^{q_n}(x)-x|<D/q_n\text{ for some }D>1
\end{equation}
 thus showing \eqref{cond:rigidity} in Theorem \ref{thm: FUcriterion} as well. The argument uses the same observation as the one used to prove \eqref{cond:cont_interval} in Proposition \ref{prop: effcriterion}, which is the following. 
 
 Let $j\in\{0,\ldots,h_{\gamma_n}^n/2\}$ and let $x\in T^j(\mathfrak I_{\Gamma}^n)$. Then $T^{q_n-j}(x)\in \mathfrak J_n$. However, $\mathfrak J_n$ is the middle level of the tower $X_n$. Hence we have that $x$
 and $T^{q_n}(x)=T^{j}(T^{q_n-j}(x))$ belong to the same level of the tower $X_n$. In particular, they belong to a continuity interval of $T$ (and as such to a continuity interval of any function continuous over exchanged intervals). Moreover, we have
 \[
 |T^{q_n}(x)-x|\le |\mathfrak J_n|\le \#\A|\mathfrak I_{\Gamma}^n|=
 \#\A\frac{1}{q_n} q_n|\mathfrak I_{\Gamma}^n|\le \frac{\#\A}{q_n}.
 \]
 Thus Conditions \eqref{cond:sequences} and \eqref{cond:cont_interval} in Proposition \ref{prop: effcriterion} are satisfied.

In the following proofs, we will check that the remaining assumptions in Proposition \ref{prop: effcriterion}, namely, Conditions \eqref{cond:continuity} and \eqref{cond:derivative}, are satisfied for the different cocycles considered in each setting. In view of the construction above this is enough to apply Proposition \ref{prop: effcriterion} and conclude the ergodicity of the skew product under consideration.

 \begin{proof}[Proof of Theorem \ref{thm: ergo}]
 We assume that $a>0$ since the other case follows symmetrically. We will apply Proposition \ref{prop: effcriterion} for $m=1$. Since $f(x)=a(x-\tfrac{1}{2})$ is continuous, \eqref{cond:continuity} in \ref{prop: effcriterion} is satisfied. We now show that \eqref{cond:derivative} is satisfied.

 First, trivially, we have
 \[
 f'(x)=a\quad\text{ for }x\in I.
 \]
 In particular
 \[
 S_{q_n} f'(x)=aq_n\quad \text{ for }x\in I.
 \]
 Recall that the tower $\bigsqcup_{i=0}^{\mathfrak h_{\Gamma}^n-1}T^i(\mathfrak I_{\Gamma}^n)$ has $\mathfrak c_n$ as its central point. By construction, the first visit time of $\mathfrak c_n$ via $T^{-1}$ to $\mathfrak J_n$ is $q_{n}/2$ or $(q_n+1)/2$. In both cases, by Lemma \ref{lem: berktrujillo}, 
 \[
 S_{q_n}f(T^{-\lfloor (q_n+1)/2\rfloor}(\mathfrak c_n))=0.
 \]
 Since $S_{q_n}f$ is continuous in $\mathfrak I_{\Gamma}^n$ and $q_n|\mathfrak I_{\Gamma}^n|<1$, by the mean value theorem we have that 
 \[
 \left|S_{q_n}f(x)\right|<a\quad \text{ for every }x\in \mathfrak I_{\Gamma}^n.
 \]
 If $x\in T^j(\mathfrak I_{\Gamma}^n)$ for $j\in \{1,\ldots,h_{\gamma_n}^n/2\}$, then by \eqref{eq: derbound} and, again, by the mean value theorem,
 \[
 \left|S_{q_n}f(x)-S_{q_n}f(T^{-j}(x)) \right|=
 \left|S_j f(T^{q_n-j}(x))-S_j f(T^{-j}(x)) \right|
 \le a D.
 \]
 Thus we get 
 \[
 \left|S_{q_n}f(x)\right|<a(D+1) \quad \text{ for every }x\in\Xi_n.
 \]
 Since the bound does not depend on $n$, \eqref{cond:derivative} is satisfied and, by Proposition \ref{prop: effcriterion}, the skew product $T_f$ is ergodic. 
 \end{proof}
 \begin{proof}[Proof of Theorem \ref{thm: uniqergo}.]
The only real difference between the proof of this result and the proof of Theorem \ref{thm: ergo} is the condition on the derivative. However, since $T$ is now uniquely ergodic, we have 
\[
\frac{1}{q_n}\sum_{i=0}^{q_n-1}f_0'\circ T^i\to 0 \text{ uniformly.}
\]
Thus for any $\varepsilon>0$ and $n$ large enough we have 
\[
(a-\varepsilon)q_n \le S_{q_n} f' (x)\le (a+\varepsilon)q_n \quad \text{ for }x\in I.
\]
By taking $\varepsilon<|a/2|$, we get the desired condition on the derivative. The rest of the proof follows analogously to the proof of Theorem \ref{thm: ergo}.
\end{proof}
\begin{proof}[Proof of Theorem \ref{thm: index}.] We use again Proposition \ref{prop: effcriterion}, this time for arbitrary $m\in\N$. We apply it to $T^{\times m}$ (which is ergodic by weak mixing) and to $f^{\times m}$. Condition \eqref{cond:continuity} in Proposition \ref{prop: effcriterion} is easily verified and Condition \eqref{cond:derivative} is satisfied in the same way as in the proof of Theorem \ref{thm: ergo}. As in the previous proofs, the ergodicity of $T_f$ follows from Proposition \ref{prop: effcriterion}.
\end{proof}

\section*{Acknowledgments} 
The authors would like to thank Krzysztof Frączek for suggesting the problem. The first author was supported by the NCN Grant 2022/45/B/ST1/00179.
 The second author was partially supported by the {\it Spanish State Research Agency}, through the {\it Severo Ochoa and María de Maeztu Program for Centers and Units of Excellence in R$\&$D} (CEX2020-001084-M) and by the {\it European Research Council} (ERC) under the European Union's Horizon 2020 research and innovation programme Grant 101154283. The third author was partially supported by Swiss National Science Foundation through Grant 200021$\_$188617/1.

\bibliographystyle{acm}
\bibliography{Bibliography.bib}
\end{document}